\documentclass[11pt]{amsart}

\usepackage{amsmath}
\usepackage{amssymb}
\usepackage{graphicx}

\newtheorem{theorem}{Theorem}[section]
\newtheorem{corollary}[theorem]{Corollary}

\newtheorem{proposition}[theorem]{Proposition}
\theoremstyle{definition}
\newtheorem{definition}[theorem]{Definition}
\newtheorem{example}[theorem]{Example}
\newtheorem{remark}[theorem]{Remark}

\numberwithin{equation}{section}

\title[On finite-dimensional set-inclusive constraint systems]{On finite-dimensional set-inclusive
constraint systems: local analysis and related optimality conditions}

\author[A. Uderzo]{Amos Uderzo}

\address[A. Uderzo]{Department of Mathematics and Applications,
University of Milano - Bicocca, Milano, Italy}
\email{{\tt amos.uderzo@unimib.it}}



\date{\today}

\newcommand{\R}{\mathbb R}
\newcommand{\N}{\mathbb N}

\newcommand{\Uball}{{\mathbb B}}
\newcommand{\Usfer}{{\mathbb S}}

\newcommand{\dom}{{\rm dom}\, }
\newcommand{\grph}{{\rm gph}\,}

\newcommand{\upinv}{{+1}}

\newcommand{\nullv}{\mathbf{0}}

\newcommand{\conv}{{\rm co}\, }

\newcommand{\cl}{{\rm cl}\, }
\newcommand{\bd}{{\rm bd}\, }
\newcommand{\inte}{{\rm int}\, }

\newcommand{\Lin}{\mathcal{L}}

\newcommand{\IGE}{{\rm IGE}\,}
\newcommand{\OP}{{\rm P}}

\newcommand{\FlD}{\widehat{\partial}}
\newcommand{\FuD}{\widehat{\partial}^+}

\newcommand{\Solv}[1]{{\mathcal Sol}(#1)}
\newcommand{\gincr}[1]{{\rm inc}_C(#1)}
\newcommand{\Eff}[1]{{\rm Eff}_C(#1)}
\newcommand{\Lev}[1]{{\rm Lev}_{{}_{#1}}}
\newcommand{\sur}[1]{{\rm sur}\, ({#1})}
\newcommand{\ndc}[1]{{#1}^{{}^\ominus}}

\newcommand{\ball}[2]{{\rm B}(#1, #2)}

\newcommand{\dist}[2]{{\rm dist}\left(#1,#2\right)}
\newcommand{\exc}[2]{{\rm exc}\left(#1,#2\right)}
\newcommand{\Tang}[2]{{\rm T}(#1;#2)} 
\newcommand{\Iang}[2]{{\rm I}(#1;#2)} 
\newcommand{\WIang}[2]{{\rm I_w}(#1;#2)} 
\newcommand{\Haus}[2]{{\rm Haus}(#1,#2)}

\newcommand{\Ncone}[2]{{\rm N}(#1;#2)}  

\newcommand{\incr}[3]{{\rm inc}_C(#1;#2;#3)}

\begin{document}

\begin{abstract}
In the present paper, some aspects of the finite-dimensional theory of set-inclusive generalized
equations are studied. Set-inclusive generalized equations are problems arising in
several contexts of optimization and variational analysis, involving multi-valued
mappings and cones. The aim of this paper is to propose an approach to the local
analysis of the solution set to such kind of generalized equations.
In particular, a study of the contingent cone to the solution set is carry
out by means of first-order approximations of set-valued mappings, which are expressed
by prederivatives. Such an approach emphasizes the role of the metric increase property
for set-valued mappings, as a condition triggering crucial error bound estimates
for the tangential description of solution sets. Some of the results obtained through this
kind of analysis are then exploited for formulating necessary optimality conditions,
which are valid for problems with constraints formalized by set-inclusive generalized
equations.
\end{abstract}

\maketitle


\begin{flushright}
{\small
 ``Cum autem universa mathesis in investigatione quantitatum incognitarum
versetur \dots", \\
L. Euler, {\it Commentatio de Matheseos sublimioris utilitate},
E. 790, O. III, 2.}
\end{flushright}


\section{Introduction} \label{Sect1}

The starting point of several investigation trends in applied nonlinear
analysis seems to be an apparent gnosiological dichotomy: from one hand,
phenomena to be studied are quantitatively described by various types
of equations; on the other hand, equations often happen to be hardly
solvable. In the case in which one solution is actually at disposal, it becomes
crucial to glean information on the geometry of the solution set
near that solution. This is true, in particular, when equations
describe the constraint system of an optimization problem.
In that context, indeed, the Lagrange's intuition, in order to
extend to constrained problems the seminal tangent approach for
searching extrema due to Fermat, essentially was the fact that an approximated
representation of the local geometry of the solution set to an equation system
(feasible region of the problem)
already suffices to trigger effective solving methods.
In a classic optimization setting, the basic first-order approximation of
the solution set to a system of smooth equations is the tangent
space to a differentiable manifold. Historically, the question of computing such a
tangent space has been successfully addressed, upon regularity
condition, even in the general setting of Banach spaces by the
Lyusternik's theorem (see \cite{Lyus34}): in fact, that result reduces a
nontrivial geometric issue to a linear algebra computation
(namely, to solve a linear system).
Later, the appearance of inequality or more complicate side-constraint
systems, which, in the Rockafellar's words `are the true hallmark of modern
optimization' (see \cite{Rock93}), led to replace the tangent
space with other one-side first-order approximations of sets and to
look for adequate algebraic representations of them, in the Lagrange
spirit. An important and well-known instance in such a trend of research is the
employment of the contingent (a.k.a. Bouligand) cone, which,
among many other local approximation of sets, emerged as a key
tool in modern approaches to the analysis of optimization problems.

The main purpose of the present paper is to start a local study of the
solution set to a certain class of generalized equations called
set-inclusive, by providing a description of its contingent cone.
After \cite{Robi79}, generalized equation is a term largely employed
in the variational analysis literature to denote a problem container,
which includes traditional equality/inequality systems, variational
inequalities and complementarity problems, fixed point and other
equilibrium conditions. In constrained optimization, constraint
systems formalized by generalized equations enable to deal with
a broad spectrum of problems.
Given a set-valued mapping $F:\R^n\rightrightarrows\R^m$,
a nonempty closed set $S\subseteq\R^n$, and a closed, convex and pointed
cone $C\subseteq\R^m$, by set-inclusive generalized equation the
problem:
$$
  \hbox{find $x\in S$ such that}\ F(x)\subseteq C,  \leqno (\IGE)
$$
is meant. The solution set to $(\IGE)$ will be denoted throughout
the paper by
$$
  \Solv{\IGE}=\{x\in S:\ F(x)\subseteq C\}.
$$
Such a class of generalized equation seems to have a very different
nature with respect to generalized equations mostly studied in
variational analysis literature, which can be put in the form
\begin{equation}     \label{in:classicge}
  g(x)\in G(x),
\end{equation}
where $g$ stands for a single-valued mapping (sometimes called base), whereas
$G$ stands for a multi-valued mapping (sometimes called field)
\footnote{In a normed space setting, with $\nullv$ denoting the null element,
$(\ref{in:classicge})$ is usually reformulated as $\nullv\in-g(x)+G(x)$.}.
Within constraint system analysis, a motivation for considering generalized
equations in the $(\IGE)$ form is discussed below.

\begin{example}[Robust approach to uncertain constraint systems]
Let us consider a cone constraint system formalized by the parametric
inclusion
\begin{equation}     \label{in:parcontsys}
  f(x,\omega)\in C,
\end{equation}
where $f:\R^n\times\Omega\longrightarrow\R^m$ is a given mapping
and $C$ is a closed, convex cone. The parameter $\omega\in\Omega$
entering the argument of $f$ describes uncertainties often occurring
in real-world optimization problems. The feasible region of such problems,
as well as their objective function, may happen to be affected by
computational and estimation errors and conditioned by unforeseeable
future events. Whereas in stochastic optimization the probability
distribution of the uncertain parameter appears among the problem
data, robust optimization assumes that no stochastic information on
the uncertain parameter is at disposal. This opens the question
on what can be considered as a solution to $(\ref{in:parcontsys})$.
An approach consists in considering as a feasible solution all
elements $x\in\R^n$ which remain feasible in every possibly occurring
scenario, that is such that
$$
   f(x,\omega)\in C,\quad\forall \omega\in\Omega.
$$
Such an approach naturally leads to introduce the robust constraining
mapping $F:\R^n\rightrightarrows\R^m$, defined as
\begin{equation}     \label{eq:robconstmap}
  F(x)=f(x,\Omega)=\{f(x,\omega):\ \omega\in\Omega\},
\end{equation}
and to consider set-inclusive generalized equations like $(\IGE)$.
\end{example}

Mappings defined as in $(\ref{eq:robconstmap})$ or, more in general,
as
$$
  F(x)=f(x,U(x))=\{f(x,y):\ y\in U(x)\},
$$
where $U:\R^n\rightrightarrows\R^m$ is a given set-valued mapping,
are called parameterized mappings after \cite{AubFra90} and are employed
in optimal control theory.

For the local study of the solution set to generalized equations as
in $(\ref{in:classicge})$, a well-developed
approach consists in replacing the nonlinear (sometimes even nonsmooth)
base term with its derivative or another kind of first-order approximation,
and then to employ stability results for getting insights into the
the solution set from the solution set of the approximated (linearized,
if possible) generalized equation.
When trying to adopt a similar approach in the case of $(\IGE)$, a
basic question to face is which approximation tool is to be used
for the term $F$.
Since the fact that $\bar x\in S$ is a solution to $(\IGE)$ involves
all elements in $F(\bar x)$, such an approximation tool should not be
based on the local behaviour of $F$ around a reference element of
its graph, but should take into account the whole set $F(\bar x)$.
For this reason, the present approach utilizes the notion of
prederivative in the sense of Ioffe (see \cite{Ioff81}). The splitting of this notion
in an outer and a inner version allows one to study separately
the question of inner and outer tangential approximation of the
solution set to a $(\IGE)$.
Another feature distinguishing the present approach is the essential
employment of the metric $C$-increase property for set-valued
mappings. This property describes a certain behaviour of
mappings that links the metric structure of the domain with the
shape of the cone $C$ appearing in $(\IGE)$. Roughly speaking,
it can be viewed as a counterpart, valid in partially ordered vector
spaces, of the so-called decrease principle for scalar functions in use
in variational analysis.
It is well known that for traditional equation systems and, to
a certain extent, for generalized equations in the form $(\ref{in:classicge})$,
open covering (and hence, metric regularity) is the main property
for mappings ensuring local solvability and, as such, it became
the key concept to achieve tangential approximations of solution sets.
In a similar manner, the metric $C$-increase property turns out
to be a key concept in order to establish a proper error bound for
$(\IGE)$ and, through such kind of estimate, to get the inner tangential
approximation of $\Solv{\IGE}$. The fundamental principle of
nonlinear analysis behind the error bound result is the Ekeland
variational principle.

The contents of the paper are arranged in the subsequent sections as
follows.
In Section \ref{Sect2}, all the analysis tools, which are needed to build
the approach here proposed and to carry out its application to constrained optimization,
are presented. Essentially, they are the semicontinuity properties of
the excess function associated with a given set-valued mapping, which
derive from corresponding semicontinuity properties for set-valued mappings;
the aforementioned metric $C$-increase property, in local or global form;
various notions of one-side tangent cones coming from nonsmooth analysis
and related properties; generalized differentiation tools for functions and
set-valued mappings, such as the Fr\'echet subdifferential, prederivatives
and fans.
Section \ref{Sect3} contains the main contribution of the paper, which
concerns the representation of tangential approximations of the solution set
to set-inclusive generalized equations. In the same section, a local
error bound for such kind of generalized equations is presented. Within
the present approach, it plays an instrumental role, but it could be
also of independent interest.
In Section \ref{Sect4}, optimization problems with constraints formalized
by set-inclusive generalized equations are considered and some of the
results achieved in Section \ref{Sect3} are exploited for establishing
necessary optimality conditions.

\vskip1cm


\section{Analysis tools} \label{Sect2}

The notation in use throughout the paper is standard. $\N$ and $\R$
denote the natural and the real number set, respectively. $\R^m_+$
denotes the nonnegative orthant in the Euclidean space $\R^m$,
whose (Euclidean) norm is indicated by $|\cdot|$. The null vector
in any Euclidean space is indicated by $\nullv$. Given an element $x$ of an
Euclidean space and a nonnegative real $r$, $\ball{x}{r}=\{z\in\R^n:\
|z-x|\le r\}$ denotes the closed ball with center $x\in\R^n$ and radius $r$.
Thus, $\Uball=\ball{\nullv}{1}$ and $\Usfer=\{v\in\Uball:\ |v|=1\}$ stand for
the unit ball and the unit sphere, respectively.
Given a subset $S$ of an Euclidean space, by $\inte S$ and $\bd S$ its topological
interior and boundary are denoted, respectively.
By $\dist{x}{S}=\inf_{z\in S}d(z,x)$ the distance of $x$ from a subset $S\subseteq
\R^n$ is denoted, with the convention that $\dist{x}{\varnothing}
=+\infty$. The $r$-enlargement of a set $S\subseteq\R^n$ is indicated by
$\ball{S}{r}=\{x\in\R^n:\ \dist{x}{S}\le r\}$. Given a pair of subsets
$S_1,\, S_2\subseteq X$, the symbol $\exc{S_1}{S_2}=\sup_{s\in S_1}
\dist{s}{S_2}$ denotes the excess of $S_1$ over $S_2$.
 Whenever $F:\R^n\rightrightarrows\R^m$ is a set-valued mapping, $\grph F$
and $\dom F$ denote the graph and the domain of $F$, respectively.
All the set-valued mappings appearing in the paper will be supposed
to take closed values, unless otherwise stated.
$\Lin(\R^n;\R^m)$ indicates the space of linear mappings acting from
$\R^n$ to $\R^m$, endowed with the operator norm $\|\cdot\|$.
Other notations will be explained contextually to their use.

Throughout the text, the acronyms l.s.c. and u.s.c. stand for
lower semicontinuous and upper semicontinuous, respectively.

\subsection{Elements of set-valued analysis}

In the following remark, several facts concerning subsets of $\R^m$, which
will be employed in the subsequent analysis, are collected.

\begin{remark}    \label{rem:vectprops}
(i) If $S\subseteq\R^m$ is convex, then for every $\alpha,\, \beta\ge 0$
one has $(\alpha+\beta)S=\alpha S+\beta S$.

(ii) If $S\subseteq\R^m$ is closed and $r>0$, then $\ball{S}{r}=
S+r\Uball$.

(iii) It is known (see \cite{PalUrb02}) that, according to the order
cancellation law, if $A$, $B$ and $C$ are convex compact subsets
of $\R^m$, then
$$
  A+B\subseteq B+C \qquad\Rightarrow\qquad A\subseteq C.
$$
Such a law can be readily extended to the case in which $C$ is only
closed and convex (possibly unbounded). Indeed, since $A+B$ is
compact, there exists $k>0$ such that
$$
  B\subseteq k\Uball \qquad\hbox{ and }\qquad
  A+B\subseteq k\Uball.
$$
Clearly, the above inclusions entail that
$$
  (B+C)\cap k\Uball\subseteq B+(C\cap 2k\Uball).
$$
This because if $x=b+c\in (B+C)\cap k\Uball$, with $b\in B$ and $c\in C$,
one has
$$
  |c|\le |b|+|b+c|\le k+k\le 2k.
$$
Therefore, it results in
$$
  A+B\subseteq (A+B)\cap k\Uball\subseteq (B+C)\cap k\Uball\subseteq
  B+(C\cap 2k\Uball).
$$
Since $A$, $B$ and $C\cap 2k\Uball$ are all convex compacts sets,
by the aforementioned order cancellation law one obtains
$$
  A\subseteq C\cap 2k\Uball\subseteq C.
$$

(iv) Let $S\subseteq\R^m$ be a closed set and $C\subseteq\R^m$ a closed,
convex cone such that for every $\epsilon>0$ it is $S\subseteq C+\epsilon\Uball$.
Then, the stronger inclusion $S\subseteq C$ actually holds.
\end{remark}

In studying the variational behaviour of set-valued mappings, a basic
tool of analysis is the excess of a set $A$ beyond a set $B$, with
$A,\, B\subseteq\R^m$, that is defined by
$$
  \exc{A}{B}=\sup_{a\in A}\dist{a}{B}=\sup_{a\in A}\inf_{b\in B}
  |a-b|.
$$

The next remark recalls some known facts concerning the behaviour
of the excess, that are needed in the subsequent section (for their
proof, see \cite{Uder18}).

\begin{remark}     \label{rem:excbehave}
(i) Let $S\subseteq\R^m$ be a nonempty set such that $\exc{S}{C}
>0$. Then, for any $r>0$ it holds $\exc{\ball{S}{r}}{C}=\exc{S}{C}+r$
(behaviour of the excess with respect to enlargements).

(ii) If $S\subseteq\R^m$, it holds $\exc{S+C}{C}=\exc{S}{C}$
(invariance of the excess under conic extension).

(iii) Let $r>0$. It holds $\exc{r\Uball}{C}=\sup_{x\in r\Uball}
\inf_{c\in C}|x-c|\le\sup_{x\in r\Uball}|x|=r$. This can not be viewed
as a special case of (i), because here $S=\{\nullv\}$ does not satisfy
$\exc{S}{C}>0$ (remember that the cone $C$ is closed).
\end{remark}

Let $F:\R^n\rightrightarrows\R^m$ be a set-valued mapping. Its domain
will be denoted by $\dom F=\{x\in\R^n:\ F(x)\ne\varnothing\}$.
Let $A\subseteq$ be a subset of $\R^m$. The upper inverse (also
called core, after \cite{AubFra90}) of $A$ through $F$ is denoted by $F^\upinv(A)=\{
x\in\R^n:\ F(x)\subseteq A\}$, whereas the lower inverse of $A$
through $F$ is denoted by $F^{-1}(A)=\{x\in\R^n:\ F(x)\cap A
\ne\varnothing\}$. Whenever $F$ is a single-valued mapping,
in which case $F^\upinv(A)=F^{-1}(A)$, the wide-spread notation
$F^{-1}(A)$ will be preferred. Recall that $F$ is said to be lower semicontinuous
(for short, l.s.c.) at $\bar x\in\R^n$ if for every open set
$O\subseteq\R^m$ such that $F(\bar x)\cap O\ne\varnothing$ there exists
$\delta_O>0$ such that
$$
  F(x)\cap O\ne\varnothing,\quad\forall x\in\ball{\bar x}{\delta_O}.
$$
$F$ is said to be Hausdorff $C$-upper semicontinuous
(for short, u.s.c.) at $\bar x\in\R^n$ if for every $\epsilon>0$
there exists $\delta_\epsilon>0$ such that
$$
  F(x)\subseteq F(\bar x)+C+\epsilon\Uball,\quad\forall
  x\in\ball{\bar x}{\delta_\epsilon}.
$$
$F$ is said to be Lipschitz with constant $\kappa\ge 0$ if
$$
  \Haus{F(x_1)}{F(x_2)}\le\kappa |x_1-x_2|,\quad\forall
  x_1,\, x_2\in\R^n.
$$

\begin{remark}     \label{rem:semexcfunct}
Given a set-valued mapping $F:\R^n\rightrightarrows\R^m$,
the following assertions linking the aforementioned semicontinuity
properties of $F$ with corresponding semicontinuity properties of
the excess function $\phi:\R^n\longrightarrow [0,+\infty]$, which
is associated with $F$ and $C$, namely
$$
  \phi(x)=\exc{F(x)}{C}
$$
will come into play:

(i) if $F$ is l.s.c. at $\bar x$, then $\phi$ is l.s.c. at $\bar x$;

(ii) if $F$ is Hausdorff $C$-u.s.c. at $\bar x$, then $\phi$ is u.s.c.
at $\bar x$.

\noindent Their proof can be found in \cite[Lemma 2.3]{Uder18}.
\end{remark}

\begin{remark}
(i) It is well known that the (Minkowski) sum of two closed sets
may fail to be a closed set, whereas the sum of a compact set
with a closed one remains closed. The former fact may cause a shortcoming
inasmuch as, given a set-valued mapping $F:\R^n\rightrightarrows\R^m$,
one would need to have $F(x)+C$ closed.
In this concern, it is worth observing that in many circumstances, even
if $F(x)$ is not compact, $F(x)+C$ still may preserve the property of being
closed. This happens, for instance, if there exists a compact set
$S_x\subseteq F(x)$ such that
$$
  S_x+C=F(x)+C.
$$
To consider such a circumstance, let us denote by $\Eff{S}$ the set
of all efficient elements of $S$ with respect to $C$, i.e.
$$
  \Eff{S}=\{x\in S:\ (x-C)\cap S=\{x\}\}.
$$
Recall that a subset $S\subseteq\R^m$ is said to enjoy the $C$-quasi domination
property provided that $\Eff{S}\ne\varnothing$ and $\Eff{S}+C=S+C$ (see, for
instance, \cite{JahHa11}).
Then, if $F$ takes values with such a property, and the set of efficient elements
of its values is compact, then $F(x)+C$ is closed.
Other sufficient conditions for $F(x)+C$ to be closed can be
formulated on the basis of specific properties of $F$.

(ii) Another known fact which is relevant to the present analysis
is that, if a set-valued mapping $F:\R^n
\rightrightarrows\R^m$ is l.s.c. at each point of $\R^n$, then
$F^\upinv(C)$ is a closed set, for every closed set $C$
(see, for instance, \cite[Lemma 17.5]{AliBor06}).
This fact makes it clear that, under the assumptions made
on the problem data of $(\IGE)$ (namely: $S$ closed and $C$
closed, convex and pointed cone), if $F$ is a l.s.c. set-valued
mapping, then the solution set $\Solv{\IGE}=S\cap F^\upinv(C)$
is a closed subset (possibly empty) of $\R^n$.
\end{remark}

\subsection{The metric $C$-increase property}

The next definition introduces the main property of set-valued mappings,
on which the proposed approach to the solution analysis of $(\IGE)$ relies.
It postulates a behaviour of mappings that links the metric structure
of the domain with the partial ordering induced on the range space
by the cone $C$ in the standard way (henceforth denoted by $\le_C$).

\begin{definition}[Metrically $C$-increasing mapping]   \label{def:metincr}
Let $S\subseteq\R^n$ be a nonempty closed set and let $C\subseteq\R^m$
be a closed, convex cone.
A set-valued mapping $F:\R^n\rightrightarrows\R^m$ is said to be:

\begin{itemize}

\item[(i)]
{\it metrically $C$-increasing} around $\bar x\in\dom F$
relative to $S$ if there exist $\delta>0$ and $\alpha>1$ such that
\begin{equation}    \label{in:defmetincr}
 \forall x\in\ball{\bar x}{\delta}\cap S,\, \forall r\in (0,\delta),\
 \exists z\in\ball{x}{r}\cap S:\ \ball{F(z)}{\alpha r}\subseteq
 \ball{F(x)+C}{r}.
\end{equation}
The quantity
$$
  \incr{F}{S}{\bar x}=\sup\{\alpha>1:\ \exists\delta>0 \hbox{ for which
  the inclusion in $(\ref{in:defmetincr})$ holds\,}\}
$$
is called {\it exact bound of metric $C$-increase} of $F$ around
$\bar x$, relative to $S$.

\item[(ii)]
{\it globally metrically $C$-increasing} if there exists
$\alpha>1$ such that
\begin{equation}    \label{in:defglometincr}
   \forall x\in\R^n,\ \forall r>0,\ \exists z\in\ball{x}{r}:\
   \ball{F(z)}{\alpha r}\subseteq  \ball{F(x)+C}{r}.
\end{equation}
The quantity
$$
  \gincr{F}=\sup\{\alpha>1: \hbox{ the inclusion in $(\ref{in:defglometincr})$
   holds\,}\}
$$
is called {\it global exact bound of metric $C$-increase} of $F$.
\end{itemize}
\end{definition}

As a comment to the above property, let us observe that the behaviour
that it describes can be regarded as a set-valued version of a
phenomenon for scalar functions, which is known in variational analysis
as decrease principle.
By this term, any condition is denoted, which ensures the existence
of a constant $\kappa>0$ such that
$$
  \inf_{x\in\ball{\bar x}{r}}\varphi(x)\le\varphi(\bar x)-\kappa r,
$$
where $\varphi:X\longrightarrow\R\cup\{+\infty\}$ is a l.s.c. and
bounded from below function defined on a proper (at least, metric)
space, $\bar x\in X$ is a reference point and $r>0$. Often, such a
condition finds a formulation in terms of Fr\'echet subdifferential
(see \cite[Theorem 3.6.2]{BorZhu05}), provided that $X$ is a Fr\'echet smooth
Banach space, or, more generally, in terms of strong slope (see \cite{AzCoLu02}),
if $X$ is a complete metric space. The decrease principle appeared
as a fundamental tool in the analysis of error bounds and solution
stability for inequalities and, as such, it plays a key role in
establishing implicit multifunction theorems (see \cite{BorZhu05}).

\begin{remark}    \label{rem:equivrefCincr}
(i) Whenever $\bar x\in\inte S$, the notion of metric $C$-increase around
$\bar x$, relative to $S$, reduces to the notion of local metric
$C$-increase around $\bar x$, as defined in \cite{Uder18}.

(ii) In the light of Remark \ref{rem:vectprops}(ii), an equivalent reformulation of
the inclusion $(\ref{in:defmetincr})$ that will be useful is
\begin{equation}    \label{def:metincrref}
 \forall x\in\ball{\bar x}{\delta}\cap S,\, \forall r\in (0,\delta),\
 \exists z\in\ball{x}{r}\cap S:\ F(z)+\alpha r\Uball\subseteq
 F(x)+C+r\Uball.
\end{equation}

(iii) Inclusions $(\ref{in:defmetincr})$, $(\ref{in:defglometincr})$ and
$(\ref{def:metincrref})$ make evident that, whenever $F$ is single-valued,
Definition \ref{def:metincr} would be never satisfied with the choice
$C=\{\nullv\}$. This entails that the approach here proposed can not
work when studying the special case of generalized equations $(\IGE)$
given by equality systems. Such a limit does not emerge in the theory
of tangential approximations of the solution set to generalized
equations in the form $(\ref{in:classicge})$: there, instead, the
equality system case inspired developments towards more general
types of equations, embedding them. On the other hand,
it must be observed that a format like $(\IGE)$ is conceived specifically
for multi-valued mappings $F$.
\end{remark}

\begin{example}     \label{ex:glometCincrmap}
Let $F:\R\rightrightarrows\R^2$ be defined by
$$
  F(x)=\{y=(y_1,y_2)\in\R^2:\ \min\{y_1,y_2\}=x\}
$$
and let $C=\R^2_+$. By a direct check of Definition \ref{def:metincr}(ii),
one can see that the set-valued mapping $F$ is globally metrically
$\R^2_+$-increasing, with $\gincr{F}=2$.
\end{example}

Further examples of classes of metrically $C$-increasing set-valued mappings,
along with verifiable conditions for detecting such property, will
be provided in the next subsection.


\subsection{Generalized differentiation tools}

Let $S\subseteq\R^n$ be a nonempty closed set and let $\bar x\in S$.
As a first-order approximation of sets the following cones will
be used:
$$
  \Tang{S}{\bar x}=\{v\in\R^n:\ \exists (v_n)_n \hbox{ with } v_n\to v,
  \ \exists (t_n)_n \hbox { with }t_n\downarrow 0:\
  \bar x+t_nv_n\in S,\ \forall n\in\N\},
$$
$$
  \Iang{S}{\bar x}=\{v\in\R^n:\ \exists\delta>0: \bar x+tv\in S,\
  \forall t\in (0,\delta)\}.
$$
and
$$
  \WIang{S}{\bar x}=\{v\in\R^n:\ \forall\epsilon>0,\ \exists
  t_\epsilon\in (0,\epsilon):\ \bar x+t_\epsilon v\in S\}.
$$
They are called the contingent cone, the feasible direction cone,
and the weak feasible direction cone to $S$ at $\bar x$, respectively
(see, for instance, \cite{AubFra90,Schi07}).
The following inclusion relation is known to hold in general
$$
  \Iang{S}{\bar x}\subseteq\WIang{S}{\bar x}\subseteq\Tang{S}{\bar x}.
$$
When, in particular, $S$ is locally convex around $\bar x$, i.e.
there exists $r>0$ such that $S\cap\ball{\bar x}{r}$ is convex,
then
$$
  \cl\Iang{S}{\bar x}=\cl\WIang{S}{\bar x}=\Tang{S}{\bar x}
$$
(see, for instance, \cite[Proposition 11.1.2(d)]{Schi07}).
The contingent cone will be the main object of study in the present
analysis. It follows from its very definition that it is determined
only by the geometric shape of a set near the reference point, namely for
any $r>0$ it is
\begin{equation}      \label{eq:locbehaTang}
  \Tang{S}{\bar x}= \Tang{S\cap\ball{\bar x}{r}}{\bar x}.
\end{equation}
Of course, whenever $S$ is a closed convex cone, one finds
$\Tang{S}{\nullv}=S$.

\begin{remark}     \label{rem:charcontcone}
Given a nonempty $S\subseteq\R^n$ and $\bar x\in S$, the following
characterization of $\Tang{S}{\bar x}$ in terms of the Dini lower
derivative of the function $x\mapsto\dist{x}{S}$ at $\bar x$
will be useful
$$
  \Tang{S}{\bar x}=\left\{v\in\R^n:\ \liminf_{t\downarrow 0}
  {\dist{\bar x+tv}{S}\over t}=0\right\}
$$
(see \cite[Proposition 11.1.5]{Schi07} and \cite{AubFra90},
where the above equality actually appears as a definition of the
contingent cone to $S$ at $\bar x $).
\end{remark}

Given a cone $C\subseteq\R^m$, the set
$$
  \ndc{C}=\{v\in\R^m:\ \langle v,c\rangle\le 0,\quad\forall c\in C\}
$$
is called (negative) dual cone of $C$. Whenever $S$ is locally convex
around $\bar x$ (and hence $\Tang{S}{\bar x}$ is convex), such an operator is connected with
the normal cone to $S$ at $\bar x$ in the sense of convex analysis by the
following well-known relation
$$
  \Ncone{S}{\bar x}=\ndc{\Tang{S}{\bar x}}.
$$

\begin{remark}     \label{rem:ndccalcul}
Given $\Lambda\in\Lin(\R^n;\R^m)$ and a pair of closed convex
cones $Q\subseteq\R^n$ and $C\subseteq\R^m$, the following
useful calculus rule holds
$$
  \ndc{(Q\cap\Lambda^{-1}(C))}=\cl(\ndc{Q}+\Lambda^\top(\ndc{C}))
$$
(see \cite[Lemma 2.4.1]{Schi07}). Notice that the equalities
\begin{equation}    \label{eq:ndcintersect}
  \ndc{(Q_1\cap Q_2)}=\cl(\ndc{Q_1}+\ndc{Q_2})
\end{equation}
and
\begin{equation}    \label{eq:invadj}
  \ndc{(\Lambda^{-1}(C))}=\Lambda^\top(\ndc{C})
\end{equation}
are special cases of the above formula. If, in particular, the
qualification condition $\inte Q_1\cap\inte Q_2\ne\varnothing$
happens to be satisfied, then formula $(\ref{eq:ndcintersect})$
takes the simpler form
\begin{equation}     \label{eq:ndcintesecteq}
   \ndc{(Q_1\cap Q_2)}=\ndc{Q_1}+\ndc{Q_2}.
\end{equation}
\end{remark}


Let $\varphi:\R^n\longrightarrow\R\cup\{\pm\infty\}$ be a
function which is finite around $\bar x\in\R^n$. Following \cite{Mord06},
the sets
$$
  \FlD\varphi(\bar x)=\left\{v\in\R^n:\ \liminf_{x\to\bar x}
  {\varphi(x)-\varphi(\bar x)-\langle v,x-\bar x\rangle\over
  |x-\bar x|}\ge 0\right\}
$$
and
$$
  \FuD\varphi(\bar x)=\left\{v\in\R^n:\ \limsup_{x\to\bar x}
  {\varphi(x)-\varphi(\bar x)-\langle v,x-\bar x\rangle\over
  |x-\bar x|}\le 0\right\}
$$
are called the Fr\'echet subdifferential of $\varphi$ at $\bar x$
and the Fr\'echet upper subdifferential of $\varphi$ at $\bar x$,
respectively. It is readily seen that, whenever $\varphi$ is (Fr\'echet)
differentiable at $\bar x$, then $\FlD\varphi(\bar x)=\FuD\varphi(\bar x)=\{
\nabla\varphi(\bar x)\}$, whereas whenever $\varphi:\R^n\longrightarrow
\R$ is convex (resp. concave), the set $\FlD\varphi(\bar x)$ (resp.
$\FuD\varphi(\bar x)$) reduces to the subdifferential (resp.
superdifferential) of $\varphi$ at $\bar x$ in the sense of convex
analysis.

\begin{remark}    \label{rem:varformFuD}
The following variational description of the Fr\'echet upper subdifferential
of $\varphi$ at $\bar x$ will be exploited in the sequel: for every $v\in\FuD
\varphi(\bar x)$ there exists a function $\sigma:\R^n\longrightarrow\R$,
differentiable at $\bar x$ and with $\varphi(\bar x)=\sigma(\bar x)$,
such that $\varphi(x)\le\sigma(x)$ for every $x\in\R^n$ and
$\nabla\sigma(\bar x)=v$ (see \cite[Theorem 1.88]{Mord06}).
\end{remark}

While cones are the basic objects for approximating sets, positively
homogeneous set-valued mappings are the basic tools for approximating
multifunctions.
Recall that a set-valued mapping $H:\R^n\rightrightarrows\R^m$ is
positively homogeneous (for short, p.h.) if $\nullv\in\ H(\nullv)$ and
$$
  H(\lambda x)=\lambda H(x),\quad\forall \lambda>0,\, \forall x\in\R^n.
$$
Within the class of p.h. set-valued mappings, fans will play a prominent
role in the present analysis.

\begin{definition}[Fan]
A set-valued mapping $H:\R^n\rightrightarrows\R^m$ is said to be
a {\it fan} if it fulfils all the following conditions:

\begin{itemize}

\item[(i)] it is p.h.;

\item[(ii)] it is convex-valued;

\item[(iii)] it holds
$$
  H(x_1+x_2)\subseteq H(x_1)+H(x_2),\quad\forall x_1,\, x_2
  \in\R^n.
$$
\end{itemize}
\end{definition}

Fans are set-valued mappings with a useful geometric structure, arising
in a large variety of contexts. It is clear that the class of all fans
acting between $\R^n$ and $\R^m$ includes, as a very special case, the space
$\Lin(\R^n;\R^m)$. Below, several examples of fans, mostly
taken from \cite{Ioff81}, are presented and discussed.

\begin{example}[Fans generated by linear mappings]     \label{ex:lingenfan}
Let ${\mathcal G}\subseteq\Lin(\R^n;\R^m)$ be a nonempty, convex and
closed set. The set-valued mapping $H:\R^n\rightrightarrows\R^m$ defined
by
$$
  H(x)=\{\Lambda x:\ \Lambda\in {\mathcal G}\}
$$
is known to be a fan. In such a circumstance, the set ${\mathcal G}$  will be
called a generator for $H$. In particular,
whenever ${\mathcal G}$ is a polytope in $\Lin(\R^n;\R^m)$, the
fan generated by ${\mathcal G}$ will be said to be finitely-generated.
For example, in the case $m=n$, one may take the class of all linear
mappings represented by $n\times n$ doubly stochastic matrices.
After the Birkhoff-von Newmann theorem, this class is known to be a
polytope, resulting from the convex hull of the permutation matrices,
which are its extreme elements (see \cite{BorLew00}).
Note that any finitely-generated fan takes compact values which are
polytopes in the range space $\R^m$. In general,
for any fan $H$ generated by linear mappings it must be $H(\nullv)=\{
\nullv\}$.
\end{example}

\begin{example}[$C$-sub/superlevel set of sub/superlinear mappings]
Let us recall that a p.h. homogeneous mapping $h:\R^n\longrightarrow\R^m$ is
said to be $C$-sublinear (resp. $C$-superlinear) if
$$
   h(x_1+x_2)\le_Ch(x_1)+h(x_2),\quad\forall x_1,\, x_2\in\R^n
$$
(resp.
$$
  h(x_1)+h(x_2)\le_C h(x_1+x_2),\quad\forall x_1,\, x_2\in\R^n \ ).
$$
With a $C$-sublinear mapping $\underline{h}:\R^n\longrightarrow\R^m$ and
with a $C$-superlinear mapping $\overline{h}:\R^n\longrightarrow\R^m$
the following fans $\Lev{\le_C\underline{h}}:\R^n\rightrightarrows\R^m$,
$\Lev{\ge_C\overline{h}}:\R^n\rightrightarrows\R^m$, and
$[\overline{h},\underline{h}]:\R^n\rightrightarrows\R^m$
are naturally associated:
$$
  \Lev{\le_C\underline{h}}(x)=\{y\in\R^m:\ y\le_C\underline{h}(x)\},
  \qquad\qquad
  \Lev{\ge_C\overline{h}}(x)=\{y\in\R^m:\ \overline{h}(x)\le_C y\},
$$
and
$$
  [\overline{h},\underline{h}](x)=\{y\in\R^m:\ \overline{h}(x)\le_C
  y\le_C\underline{h}(x)\}.
$$
It is interesting to note that $\Lev{\le_C\underline{h}}$ and
$\Lev{\ge_C\overline{h}}$ are somehow connected with fans considered
in Example \ref{ex:lingenfan}, via the notion of support of a $C$-sublinear
mapping, i.e. the set
$$
  \partial\underline{h}=\{\Lambda\in\Lin(\R^n;\R^m):\ \Lambda x\le_C
  \underline{h}(x),\ \forall x\in\R^n\}.
$$
This is true provided that the partial ordering $\le_C$ is Dedekind complete
(any bounded from above set admits a least upper bound, i.e. supremum) and norm
monotone ($C-\sup\{x, -x\}\le_CC-\sup\{z, -z\}$ implies $|x|\le |z|$),
as it happens, for instance, if $C=\R^m_+$.
Indeed, as a consequence of the Hahn-Banach-Kantorovich theorem (see \cite{Kuta79})
a continuous sublinear mapping $\underline{h}:\R^n\longrightarrow\R^m$ can be pointwise
represented as
$$
  \underline{h}(x)=C-\max_{\Lambda\in\partial\underline{h}}\Lambda x,
$$
where $C-\max$ denotes the maximum of a subset of $\R^m$, with respect to the
partial ordering $\le_C$. Thus, one sees that $y\in\Lev{\le_C\underline{h}}(x)$
iff there exists $\Lambda\in\partial\underline{h}$ such that $y\le_C\Lambda x$,
that is $y\in\Lambda x-C$. Consequently, the representation
$$
  \Lev{\le_C\underline{h}}(x)=\{y\in\R^m:\ y\in\Lambda x-C,\ \Lambda\in
  \partial\underline{h}\}
$$
holds. An analogous representation holds true for $\Lev{\ge_C\overline{h}}$,
where the support must be replaced by the upper support of a $C$-superlinear
mapping.
\end{example}

\begin{example}    \label{ex:singularfan}
The set-valued mapping $H:\R\rightrightarrows\R$, defined by
$$
  H(x)=\left\{\begin{array}{ll}
     -\R_+, & \hbox{ if } x<0, \\
     \R,   & \hbox{ if } x=0, \\
     \R_+ & \hbox{ if } x>0,
      \end{array}\right.
$$
is a fan. Since it is $H(0)=\R$, it is clear that $H$ can not be
generated by any set ${\mathcal G}\subseteq\Lin(\R;\R)$. Besides,
since if $h:\R\longrightarrow\R$ is p.h., then $h(0)=0$, one has that $H$
can not be represented as a $\R_+$-sub/superlevel set associated with
some $\R_+$-sub/superlinear mapping.
\end{example}

According to the present approach of analysis, the upper inverse image
of $C$ through a given fan will be a key element to express the
tangential approximation of $\Solv{\IGE}$. In this perspective,
the next remark gathers some elementary algebraic/topological
properties of such a set.

\begin{remark}   \label{rem:Hucone}
(i) It is plain to see that if $H:\R^n\rightrightarrows\R^m$ is a fan
and $C\subseteq\R^m$ is a closed convex cone, then the set $H^\upinv(C)$
is a convex cone (possibly empty). Notice that, in general, $H^\upinv(C)$ may happen to be
not closed. For example, if taking $C=\R_+$ and such a fan $H:\R
\rightrightarrows\R$ as defined in Example \ref{ex:singularfan},
one finds $H^\upinv(C)=(0,+\infty)$ (consistently, $H$ fails to be
l.s.c. at $0$).

(ii) It is worth noting that, in the case of a fan generated by a
set ${\mathcal G}\subseteq\Lin(\R^n;\R^m)$, it results in
$$
  H^\upinv(C)=\bigcap_{\Lambda\in{\mathcal G}}\Lambda^{-1}(C).
$$
As an immediate consequence of the last equality, one deduces that
the convex cone $H^\upinv(C)$ is closed, whenever $H$ is a fan generated
by linear mappings. Furthermore, if a fan $H$ is finitely-generated, i.e.
${\mathcal G}=\conv\{\Lambda_1,\dots,\Lambda_p\}$, with $\Lambda_i\in
\Lin(\R^n;\R^m)$, for $=1,\dots,p$, then it results in
$$
  H^\upinv(C)=\bigcap_{i=1}^p\Lambda_i^{-1}(C).
$$
In this case, each set $\Lambda_i^{-1}(C)$ turns out to be polyhedral,
provided that $C$ is so, and therefore $ H^\upinv(C)$ inherits a
polyhedral cone structure.

(iii) Whenever $H:\R^n\rightrightarrows\R^m$ is a fan generated by
a bounded set ${\mathcal G}\subseteq\Lin(\R^n;\R^m)$, it turns out
to be Lipschitz. More precisely, if $l=\sup\{\|\Lambda\|:\ \Lambda\in
{\mathcal G}\}<+\infty$, it holds
$$
  \Haus{H(x_1)}{H(x_2)}\le l |x_1-x_2|,\quad\forall x_1,\, x_2\in\R^n.
$$
Indeed, since for any $y\in\R^m$ it is
$$
  \dist{y}{H(x_2)}=\inf_{\Lambda\in{\mathcal G}}|y-\Lambda x_2|,
$$
then, if $y=\tilde\Lambda x_1$ for some $\tilde\Lambda\in{\mathcal G}$,
it results in
$$
  \dist{\tilde\Lambda x_1}{H(x_2)}=\inf_{\Lambda\in{\mathcal G}}
  |\tilde\Lambda x_1-\Lambda x_2|\le |\tilde\Lambda x_1-\tilde\Lambda x_2|
  \le\|\tilde\Lambda\| |x_1-x_2|.
$$
It follows
\begin{eqnarray*}
  \exc{H(x_1)}{H(x_2)} &=& \sup_{y\in H(x_1)}\dist{y}{H(x_2)}=
  \sup_{\Lambda\in{\mathcal G}}\dist{\Lambda x_1}{H(x_2)}   \\
  &\le & \sup_{\Lambda\in{\mathcal G}}\|\Lambda\| |x_1-x_2|
  \le l |x_1-x_2|,\quad\forall x_1,\, x_2\in\R^n.
\end{eqnarray*}
In particular, all finitely-generated fans are Lipschitz continuous
and, if ${\mathcal G}=\conv\{\Lambda_1,\dots,\Lambda_p\}$, it results
in $l=\max_{i=1,\dots,p}\|\Lambda_i\|$.
\end{remark}

The aforementioned features motivate the choice of fans as a
possible tool for approximating more general and less structured
set-valued mappings.

In view of the employment of the metric $C$-increase property
in the present approach,
the next proposition provide conditions for a fan to be globally
metrically $C$-increasing.

\begin{proposition}     \label{pro:sufcondmincr}
Let $H:\R^n\rightrightarrows\R^m$ be a fan. If
\begin{equation}    \label{in:scgmetincr}
   \exists\eta>0\ \exists u\in\Uball:\ H(u)+\eta\Uball\subseteq C,
\end{equation}
then $H$ is globally metrically $C$-increasing and $\gincr{H}\ge\eta+1$.
Conversely, if the fan $H:\R^n\rightrightarrows\R^m$ takes compact values,
then condition $(\ref{in:scgmetincr})$ is also necessary for $H$ to be
globally metrically $C$-increasing.
\end{proposition}

\begin{proof}
Take arbitrary $x\in\R^n$ and $r>0$. Letting $u\in\Uball$ and $\eta>0$ as in
condition $(\ref{in:scgmetincr})$ and setting $z=x+ru$, one has that
$z\in\ball{x}{r}$ and obtains
\begin{eqnarray*}
  H(z)+(\eta+1)r\Uball &\subseteq & H(x)+rH(u)+\eta r\Uball+r\Uball
  =H(x)+r(H(u)+\eta\Uball)+r\Uball  \\
  &\subseteq & H(x)+C+r\Uball.
\end{eqnarray*}
According to Definition \ref{def:metincr}(ii) and Remark \ref{rem:equivrefCincr},
this proves that $H$ is globally metrically $C$-increasing.

Conversely, observe first of all that if $H$ takes compact values, then
it must be $H(\nullv)=\{\nullv\}$. Indeed, as $H$ is p.h., one has
$\lambda H(\nullv)=H(\lambda\nullv)=H(\nullv)$, for any $\lambda>0$, so
$H(\nullv)$ is a cone, but $\{\nullv\}$ is the only compact cone.
Now, if $H$ is globally metrically $C$-increasing, for some $\alpha\in(1,
\gincr{H})$, taking $x=\nullv$ and $r=1$, there exists $v\in\Uball$ such that
$$
   H(v)+\alpha\Uball\subseteq H(\nullv)+C+\Uball,
$$
and hence, on account of Remark \ref{rem:vectprops}(i), it holds
$$
   H(v)+(\alpha-1)\Uball+\Uball\subseteq H(\nullv)+C+\Uball
   =C+\Uball.
$$
Since $H(v)+(\alpha-1)\Uball$ and $\Uball$ are compact convex sets,
by virtue of what noticed in Remark \ref{rem:vectprops}(iii) (extended
order cancellation law), from the last inclusion one obtains condition
$(\ref{in:scgmetincr})$, with $\eta=\alpha-1>0$.
\end{proof}

\begin{remark}
(i) Notice that the condition for metric $C$-increase expressed
by $(\ref{in:scgmetincr})$ requires that $\inte C\ne\varnothing$.
As a consequence, whenever working with finitely generated fans,
which are supposed to be globally metrically $C$-increasing, one
is forced to assume that $\inte C\ne\varnothing$.

(ii) Condition $(\ref{in:scgmetincr})$ may be read in terms of
``positivity". Take into account that, with reference to the
partial order induced by $C$, the elements in $C$ are the positive
ones. Thus, condition $(\ref{in:scgmetincr})$ postulates the
existence of a direction, along which $H$ takes strictly positive
values only.
\end{remark}

\begin{example}
According to Definition \ref{def:metincr}, the fan $H$ considered in
Example \ref{ex:singularfan} fails to be metrically $\R_+$-increasing
around each point of $\R$, relative to $S=\R$. Observe that
condition $(\ref{in:scgmetincr})$ is not satisfied.
\end{example}

From condition $(\ref{in:scgmetincr})$ one can derive a sufficient
condition for the global metric $C$-increase property, which is specific for
fans generated by regular linear mappings. Recall that if $\Lambda\in
\Lin(\R^n;\R^m)$ is regular (i.e. onto, or equivalently it is an epimorphism),
then there exists $\eta>0$ such that $\Lambda\Uball\supseteq\eta\Uball$.
The quantity
$$
  \sur{\Lambda}=\sup\{\eta>0:\ \Lambda\Uball\supseteq\eta\Uball\}
$$
is called exact openness bound of $\Lambda$ and is used to provide a measure
of the regularity (openness or covering) of $\Lambda$.
For more details on the notion of openness of linear mappings
the reader is referred to \cite[Section 1.2.3]{Mord06}. In particular,
for exact estimates of $\sur{\Lambda}$, see \cite[Corollary 1.58]{Mord06}.

\begin{corollary}
Let $H:\R^n\rightrightarrows\R^m$ be a fan generated by ${\mathcal G}
\subseteq\Lin(\R^n;\R^m)$. Suppose that
$$
  \inf_{\Lambda\in{\mathcal G}}\sur{\Lambda}>0.
$$
If
$$
  \inte\left(\bigcap_{\Lambda\in{\mathcal G}}\Lambda^{-1}(C)\right)
  \ne\varnothing,
$$
then $H$ is globally metrically $C$-increasing.
\end{corollary}

\begin{proof}
By hypothesis, there exist $u\in\R^n$ and $\epsilon>0$ such that
$$
  u+\epsilon\Uball\subseteq\bigcap_{\Lambda\in{\mathcal G}}
  \Lambda^{-1}(C).
$$
Notice that it is possible to assume that $u\ne\nullv$, because
if it is $\nullv\in\inte(\cap_{\Lambda\in{\mathcal G}}\Lambda^{-1}(C))$,
that is $\epsilon\Uball\subseteq\cap_{\Lambda\in{\mathcal G}}
\Lambda^{-1}(C)$, then there must exist $x\ne\nullv$ such that
$$
  \Lambda x\in C\qquad\hbox{ and }\qquad \Lambda(-x)\in C,
  \quad\forall \Lambda\in{\mathcal G}.
$$
Since $C$ is a pointed cone, the above inclusions imply
$\Lambda x=\nullv$, so $x\in\Lambda^{-1}(C)$ for every $\Lambda
\in{\mathcal G}$. Since $\cap_{\Lambda\in{\mathcal G}}\Lambda^{-1}(C)$
is a cone, it is possible to assume furthermore that $u\in\Uball$.
Letting $0<\eta<\inf_{\Lambda\in{\mathcal G}}\sur{\Lambda}$, since
$\sur{\Lambda}>\eta$ for every $\Lambda\in{\mathcal G}$, one has
$$
  \Lambda(\epsilon\Uball)\supseteq\epsilon\eta\Uball,\quad
  \forall\Lambda\in{\mathcal G}.
$$
Therefore, it holds
$$
  \Lambda u+\epsilon\eta\Uball\subseteq\Lambda(u+\epsilon\Uball)
  \subseteq C, \quad\forall\Lambda\in{\mathcal G}.
$$
According to the definition of $H$, it follows that $H(u)+\epsilon
\eta\Uball\subseteq C$, so the sufficient condition $(\ref{in:scgmetincr})$
for a fan to be globally metrically $C$-increasing is satisfied. The thesis
follows from Proposition \ref{pro:sufcondmincr}.
\end{proof}


In order to utilize p.h. set-valued mappings and, in particular, fans
as an approximation tool for general multivalued mappings, a concept
of differentiation is needed.  Among various proposals extending
differential calculus to a set-valued context, motivated by the
specific features of the subject under study, the notion of prederivative
as can be found in \cite{Ioff81} is here employed. Such a
notion has been recently investigated for different purposes
also in \cite{GaGeMa16,Pang11}.

\begin{definition}[Prederivative]      \label{def:prederiv}
Let $F:\R^n\rightrightarrows\R^m$ be
a set-valued mapping and let $\bar x\in\dom F$. A p.h. set-valued mapping
$H:\R^n\rightrightarrows\R^m$ is said to be a

\begin{itemize}

\item[(i)] {\it outer prederivative} of $F$ at $\bar x$ if for every $\epsilon>0$
there exists $\delta>0$ such that
$$
  F(x)\subseteq F(\bar x)+H(x-\bar x)+\epsilon |x-\bar x|\Uball,
  \quad\forall x\in\ball{\bar x}{\delta};
$$

\item[(ii)] {\it inner prederivative} of $F$ at $\bar x$ if for every $\epsilon>0$
there exists $\delta>0$ such that
$$
  F(\bar x)+H(x-\bar x)\subseteq F(x)+\epsilon |x-\bar x|\Uball,
  \quad\forall x\in\ball{\bar x}{\delta};
$$

\item[(iii)] {\it prederivative} of $F$ at $\bar x$ if $H$ is both, an
outer and an inner prederivative of $F$ at $\bar x$.
\end{itemize}
\end{definition}

It is clear that, whenever a set-valued mapping $F$ happens to be
single-valued in a neighbourhood of $\bar x$ and $H$ is a p.h. mapping,
then all cases (i), (ii), and (iii) in Definition \ref{def:prederiv}
coincide with the notion of Bouligand derivative (a.k.a. B-derivative),
as introduced in \cite{Robi91}. In particular, if $H\in\Lin(\R^n;\R^m)$
the above three notions collapse to the notion of Fr\'echet differentiability
for mappings.
In full analogy with the calculus for single-valued smooth mappings, in the
current context a strict variant of the notion of prederivative, which will
be employed in the sequel, may be formulated following \cite{GaGeMa16,Pang11}.

\begin{definition}[Strict prederivative]     \label{def:strictprederiv}
Let $F:\R^n\rightrightarrows\R^m$ be
a set-valued mapping and let $\bar x\in\dom F$. A p.h. set-valued mapping
$H:\R^n\rightrightarrows\R^m$ is said to be a {\it strict prederivative}
of $F$ at $\bar x\in\dom F$ if for every $\epsilon>0$ there exists $\delta>0$
such that
$$
  F(x_2)\subseteq F(x_1)+H(x_2-x_1)+\epsilon |x_2-x_1|\Uball,
  \quad\forall x_1,\, x_2\in\ball{\bar x}{\delta}.
$$
\end{definition}

An articulated discussion on the existence of prederivatives and strict
prederivative, on calculus rules and connections with regularity properties, can
be found in \cite{GaGeMa16,Pang11}.

\begin{example}[Uniformly smooth mappings] Let $\Xi$ be a nonempty
parameter set and let $f:\R^n\times\Xi\longrightarrow
\R^m$ be a mapping which is smooth at a point $\bar x\in\R^n$ with respect
to $x$, uniformly in $\xi$, in the sense that for every $\xi\in\Xi$
there exists $\Lambda_\xi\in\Lin(\R^n;\R^m)$ such that
$$
  f(x,\xi)=f(\bar x,\xi)+\Lambda_\xi(x-\bar x)+o_\xi(|x-\bar x|)
$$
and
$$
  \lim_{x\to\bar x}\sup_{\xi\in\Xi}{|o_\xi(|x-\bar x|)|\over|x-\bar x|}
  =0.
$$
Define a set-valued mapping $F:\R^n\rightrightarrows\R^m$ as
$$
   F(x)=f(x,\Xi).
$$
It is possible to show that if the set ${\mathcal G}=\{\Lambda_\xi:\
\xi\in\Xi\}$ is closed and convex in $\Lin(\R^n;\R^m)$, then the fan
$H:\R^n\rightrightarrows\R^m$ generated by the set ${\mathcal G}$
is an outer prederivative of $F$ at $\bar x$.
\end{example}

\begin{remark}
The reader should notice that notion in Definition \ref{def:prederiv}(ii) and
consequently in Definition \ref{def:prederiv}(iii) are different from
the notion of inner $T$-derivative and of $T$-derivative, respectively,
as proposed in \cite{Pang11}. This because the term $H(x-\bar x)$
appears in the left side of the inclusion in Definition \ref{def:prederiv}
(ii). Such a choice entails that a strict prederivative in the
sense of Definition \ref{def:strictprederiv} could fail to be a
prederivative of the same set-valued mapping. This fact is in
contrast with what happens for $T$-derivative e strict $T$-derivative
and therefore it causes a shortcoming in the resulting theory.
Nevertheless, such a choice seems to be unavoidable in order to
obtain the outer tangential approximation of $\Solv{\IGE}$,
where the values of $H$ must be included in $\Tang{C}{\bar y}$,
for $\bar y\in F(\bar x)$ (see the proof of Theorem \ref{thm:otanapproxsol}).
In this concerns, it could be relevant to observe that in
\cite[Definition 9.1]{Ioff81} (where $F$ is single-valued), the
p.h. term appears in the left side of the inclusion defining the
inner prederivative.
\end{remark}

The next result shows how local approximations expressed by
certain prederivatives can be exploited to formulate a condition
for the metric $C$-increase property of set-valued mappings
around a reference point, relative to a given set.

\begin{proposition}[Metric $C$-increase via strict prederivative]
\label{pro:metincrsopred}
Let $F:\R^n\rightrightarrows\R^m$ be a set-valued mapping, let
$S\subseteq\R^n$ be a closed set and let $\bar x\in\dom F\cap S$.
Suppose that:

(i) $F$ admits a strict prederivative $H:\R^n\rightrightarrows\R^m$
at $\bar x$;

(ii) there exist $\eta>0$ and $\delta>0$ such that
\begin{equation}     \label{in:locsufcondmCincr}
   \forall x\in\ball{\bar x}{\delta}\cap S\
   \exists u\in\Uball\cap\WIang{S}{x}:\ H(u)+\eta\Uball\subseteq C.
\end{equation}
\noindent Then, $F$ is metrically $C$-increasing around $\bar x$
relative to $S$ with
\begin{equation}
  \incr{F}{S}{\bar x}\ge\eta+1.
\end{equation}
\end{proposition}

\begin{proof}
Fix $\epsilon\in (0,\min\{1,\,\eta\})$. According to hypothesis (i), there exists
$\delta_\epsilon>0$ such that
\begin{equation}    \label{in:scincrstroutpreder}
  F(x_1)\subseteq F(x_2)+H(x_1-x_2)+\epsilon |x_1-x_2|\Uball,
  \quad\forall x_1,\, x_2\in\ball{\bar x}{\delta_\epsilon}.
\end{equation}
Choose $\delta_*\in (0,\min\{\delta,\,\delta_\epsilon/3\})$ and take arbitrary
$x\in\ball{\bar x}{\delta_*}\cap S$ and $r\in (0,\delta_*)$.

Since $x\in \ball{\bar x}{\delta}$, by virtue of hypothesis (ii)
there exists $u\in\Uball\cap\WIang{S}{x}$ such that $H(u)+\eta\Uball\subseteq C$.
Since $u\in\WIang{S}{x}$, corresponding to $r$ there must exist
$t_*\in (0,r)$, such that $x+t_*u\in S$.
Thus, if defining $z=x+t_*u$, one has
$$
  |z-\bar x|\le |z-x|+|x-\bar x|\le r+\delta_*<{2\over 3}\delta_\epsilon.
$$
This means that $z\in\ball{\bar x}{\delta_\epsilon}$, so it is possible
to apply inclusion $(\ref{in:scincrstroutpreder})$, with $x_1=z$ and
$x_2=x$. Consequently, recalling Remark \ref{rem:vectprops}(i) and (ii),
one obtains
\begin{eqnarray*}
  F(z)+(\eta+1-\epsilon)r\Uball &\subseteq& F(x)+t_*H(u)+\epsilon t_*\Uball+
  (\eta+1-\epsilon)r\Uball \\
  &\subseteq & F(x)+r[H(u)+\eta\Uball]+\epsilon r\Uball+(1-\epsilon)r\Uball  \\
  & \subseteq & F(x)+rC+r\Uball=F(x)+C+r\Uball.
\end{eqnarray*}
As $z\in\ball{x}{r}\cap S$ and it is $\eta+1-\epsilon>1$, the last
inclusion shows that $F$ is metrically $C$-increasing around
$\bar x$ relative to $S$. The
arbitrariness of $\epsilon>0$ enables one to get the quantitative estimate
of $\incr{F}{S}{\bar x}$ in the thesis.
\end{proof}

Condition $(\ref{in:locsufcondmCincr})$ can be regarded as a localization
of condition $(\ref{in:scgmetincr})$. This shows that the approximation
apparatus based on prederivatives transforms properties of approximations
into corresponding properties of the mappings to be approximated, as it
happens with classical differential calculus and certain specific properties
such as metric regularity (see \cite{Ioff81,Mord06}).

\vskip1cm


\section{Tangential approximation of solution sets} \label{Sect3}

The first result exposed in the current section is a refinement of
an error bound estimate, which was recently established in a more
general setting (see \cite[Theorem 4.3]{Uder18}). Even though its
proof follows the same lines as in \cite{Uder18}, apart from a few
adjustments due to the presence of the set $S$, here it is provided in
full detail for the sake of completeness.

\begin{theorem}[Local error bound under metric $C$-increase]   \label{thm:erbometincr}
Let $F:\R^n\rightrightarrows\R^m$ be a set-valued mapping, let $S$
be a closed set defining a $(\IGE)$, and let $\bar x\in\Solv{\IGE}$.
Suppose that:

(i) $F$ is l.s.c. in a neighbourhood of $\bar x$ and Hausdorff $C$-u.s.c.
at $\bar x$;

(ii) $F$ is metrically $C$-increasing around $\bar x$, relatively to $S$.

\noindent Then, for every $\alpha\in (1,\incr{F}{S}{\bar x})$ there exists
$\delta_\alpha>0$ such that
\begin{equation}     \label{in:locerboCincr}
  \dist{x}{\Solv{\IGE}}\le {\exc{F(x)}{C}\over\alpha-1},\quad\forall
  x\in\ball{\bar x}{\delta_\alpha}\cap S.
\end{equation}
\end{theorem}

\begin{proof}
Consider the function $\phi$, defined by
$$
  \phi(x)=\exc{F(x)}{C}.
$$
In the light of Remark \ref{rem:semexcfunct},
since $F$ is l.s.c. in $\ball{\bar x}{\delta_0}$, for some $\delta_0>0$,
so is $\phi$.
By hypothesis (ii), fixed $\alpha\in (1,\incr{F}{S}{\bar x})$ there exists
$\delta_1>0$ such that
\begin{equation}    \label{in:hyplocincrdelta1}
   \forall x\in\ball{\bar x}{\delta_1}\cap S,\, \forall r\in (0,\delta_1),\
 \exists z\in\ball{x}{r}\cap S:\ \ball{F(z)}{\alpha r}\subseteq
 \ball{F(x)+C}{r}.
\end{equation}
Furthermore, since $F$ is Hausdorff $C$-u.s.c. at $\bar x$,
$\phi$ turns out to be continuous at $\bar x$. As a result, there exists
$\delta_2>0$ such that
\begin{equation}    \label{in:phicontdelta2}
  \phi(x)-\phi(\bar x)=\phi(x)\le {\delta_1\over 2},\quad\forall
  x\in\ball{\bar x}{\delta_2}.
\end{equation}
Define
$$
  \delta_\alpha={1\over 4}\min\{\delta_0,\,\delta_1,\,\delta_2\}.
$$
To show the thesis, fix an arbitrary $x_0\in [\ball{\bar x}{\delta_\alpha}
\cap S]\backslash\Solv{\IGE}$, so $\phi(x_0)>0$.
Notice that, with the above choice of the value of $\delta_\alpha$, it is
$$
  \ball{x_0}{\delta_\alpha}\cap S\subseteq\ball{\bar x}{2\delta_\alpha}
  \subseteq\ball{\bar x}{\delta_0/2}.
$$
Thus $\phi$ is l.s.c. on the closed set $\ball{x_0}{\delta_\alpha}\cap S$
and obviously bounded from below. As it is clearly $\phi(x_0)\le\inf_{x\in
\ball{x_0}{\delta_\alpha}\cap S}\phi(x)+\phi(x_0)$, it is possible to invoke the
Ekeland Variational Principle. Accordingly, corresponding to
$\lambda={\phi(x_0)\over \alpha-1}$, there exists $x_\lambda\in
\ball{x_0}{\delta_\alpha}\cap S$ with the following properties:
\begin{equation}      \label{in:EVPloc1}
    \phi(x_\lambda)\le\phi(x_0)-(\alpha-1)d(x_\lambda,x_0),
\end{equation}
\begin{equation}     \label{in:EVPloc2}
    d(x_\lambda,x_0)\le{\phi(x_0)\over \alpha-1}
\end{equation}
and
\begin{equation}      \label{in:EVPloc3}
   \phi(x_\lambda)<\phi(x)+(\alpha-1)d(x,x_\lambda),\quad\forall
   x\in[\ball{x_0}{\delta_\alpha}\cap S]\backslash\{x_\lambda\}.
\end{equation}
If $\phi(x_\lambda)=0$, then $x_\lambda\in\Solv{\IGE}$, so one
obtains on account of inequality $(\ref{in:EVPloc2})$
$$
   \dist{x_0}{\Solv{\IGE}}\le d(x_0,x_\lambda)\le{\exc{F(x_0)}{C}\over\alpha-1}.
$$
If $\phi(x_\lambda)>0$, since $x_\lambda\in\ball{\bar x}{\delta_2}$
because $\ball{x_0}{\delta_\alpha}\subseteq\ball{\bar x}{2\delta_\alpha}
\subseteq\ball{\bar x}{\delta_2/2}$,
one gets as a consequence of inequalities $(\ref{in:EVPloc1})$ and
$(\ref{in:phicontdelta2})$
$$
  \phi(x_\lambda)\le\phi(x_0)\le{\delta_1\over 2}.
$$
On the other hand, it holds
$$
  d(x_\lambda,\bar x)\le d(x_\lambda,x_0)+d(x_0,\bar x)\le
  \delta_\alpha+\delta_\alpha<\delta_1.
$$
This makes it possible to invoke property $(\ref{in:hyplocincrdelta1})$,
with $x=x_\lambda$ and $r=\phi(x_\lambda)$.
Thus, there must exist $z\in\ball{x_\lambda}{\phi(x_\lambda)}\cap S$
such that
$$
  F(z)+\alpha\phi(x_\lambda)\Uball\subseteq F(x_\lambda)+C+
  \phi(x_\lambda)\Uball.
$$
Notice that, by properties of the excess recalled in Remark \ref{rem:excbehave}(i)
and (ii), it holds
\begin{eqnarray*}
  \phi(z) &=&\exc{F(z)}{C}=\exc{\ball{F(z)}{\alpha\phi(x_\lambda)}}{C}
      -\alpha\phi(x_\lambda)   \\
      & \le& \exc{\ball{F(x_\lambda)+C}{\phi(x_\lambda)}}{C}
      -\alpha\phi(x_\lambda)  \\
      &=&\exc{F(x_\lambda)+C}{C}+\phi(x_\lambda)-\alpha\phi(x_\lambda) \\
      &=&\exc{F(x_\lambda)}{C}+(1-\alpha)\phi(x_\lambda)
      =(2-\alpha)\phi(x_\lambda).
\end{eqnarray*}
As a consequence, whenever it happens that $\alpha>2$, the last
inequalities lead to an absurdum, thereby showing that it must be
$\phi(x_\lambda)=0$. So, henceforth it is possible to assume
without loss of generality that $\alpha\in (1,2]$.

In such a circumstance, the following two cases must be considered.

\framebox[1.1\width]{Case $d(z,x_0)\le\delta_\alpha$:} As it is $z\in
[\ball{x_0}{\delta_\alpha}\cap S]\backslash\{x_\lambda\}$, the
inequality $(\ref{in:EVPloc3})$ can be exploited, so, by recalling
the above estimate of $\phi(z)$, one obtains
$$
  \phi(x_\lambda)<\phi(z)+(\alpha-1)d(z,x_\lambda)\le
  (2-\alpha)\phi(x_\lambda)+(\alpha-1)d(z,x_\lambda)\le\phi(x_\lambda)
$$
which leads to an absurdum. Therefore, one is forced to conclude
that $\phi(x_\lambda)=0$.

\framebox[1.1\width]{Case $d(z,x_0)>\delta_\alpha$:} By recalling
inequalities $(\ref{in:EVPloc1})$ and $(\ref{in:EVPloc2})$, one
finds
\begin{eqnarray*}
  d(z,x_0) &\le& d(z,x_\lambda)+d(x_\lambda,x_0)\le \phi(x_0)
  -(\alpha-1)d(x_\lambda,x_0)+d(x_\lambda,x_0) \\
   &=& \phi(x_0)+(2-\alpha)d(x_\lambda,x_0)\le\phi(x_0)+
   {2-\alpha\over \alpha-1}\phi(x_0)
   ={\phi(x_0)\over \alpha-1}.
\end{eqnarray*}
Since it is $\bar x\in\Solv{\IGE}$, it follows
$$
  \dist{x_0}{\Solv{F,C}}\le d(x_0,\bar x)\le\delta_\alpha
  \le{\phi(x_0)\over \alpha-1}.
$$
This completes the proof.
\end{proof}

\begin{example}[Error bound failure]
Consider the set-valued mapping $F:\R\rightrightarrows\R$ defined by
$$
  F(x)=[-x^2,+\infty),
$$
and take $S=\R$, $C=\R_+$ and $\bar x=0$. With these data, the resulting
$(\IGE)$ evidently admits $\{0\}$ has a solution set. Therefore, one has
$$
  \dist{x}{\Solv{\IGE}}=|x|,\quad\forall x\in\R.
$$
On the other hand, one sees that it is
$$
  \exc{F(x)}{\R_+}=x^2,\quad\forall x\in\R.
$$
As a consequence, for any $\alpha>1$, the error bound inequality
$$
  \dist{x}{\Solv{\IGE}}=|x|\le {x^2\over\alpha-1}=
  {\exc{F(x)}{\R_+}\over\alpha-1}
$$
fails to hold in any interval $(-\delta_\alpha,\delta_\alpha)$,
whatever the value of $\delta_\alpha>0$ is. Observe that
$F$ is both l.s.c. in a neighbourhood of $0$ and Hausdorff
$\R_+$-u.s.c. at $0$, so hypothesis (i) of Theorem \ref{thm:erbometincr}
is fulfilled. Instead, $F$ is not metrically $\R_+$-increasing around $0$,
relative to $\R$ (in other terms, locally metrically $\R_+$-increasing
around $0$).
The present example thus illustrates the essential role played by the metric
$C$-increase property for the validity of the error bound $(\ref{in:locerboCincr})$.
\end{example}

The main result of the paper, about a tangential approximation of
$\Solv{\IGE}$ near one of its elements, is established below.

\begin{theorem}[Inner tangential approximation under $C$-increase]    \label{thm:itanapproxsol}
With reference to problem $(\IGE)$, let $\bar x\in\Solv{\IGE}$. Suppose
that:

(i) $F$ is l.s.c. in a neighbourhood of $\bar x$ and Hausdorff $C$-u.s.c.
at $\bar x$;

(ii) $F$ is metrically $C$-increasing around $\bar x$ relative to $S$;

(iii) $F$ admits $H:\R^n\rightrightarrows\R^m$ as an outer prederivative
at $\bar x$.

\noindent Then, the following inclusion holds
\begin{equation}     \label{in:itanapproxsol}
  H^\upinv(C)\cap\WIang{S}{\bar x}\subseteq \Tang{\Solv{\IGE}}{\bar x}.
\end{equation}
If, in addition,

(iv) the outer prederivative $H$ of $F$ at $\bar x$ is Lipschitz,

\noindent the following stronger inclusion holds
\begin{equation}     \label{in:tangapproxsol}
  H^\upinv(C)\cap\Tang{S}{\bar x}\subseteq \Tang{\Solv{\IGE}}{\bar x}.
\end{equation}
\end{theorem}

\begin{proof}
Take an arbitrary $v\in H^\upinv(C)\cap\WIang{S}{\bar x}$. If $v=\nullv$,
then it is obviously $v\in\Tang{\Solv{\IGE}}{\bar x}$. So, let us
suppose henceforth $v\ne\nullv$. Observe that, since
$H^\upinv(C)$, $\WIang{S}{\bar x}$ and $\Tang{\Solv{\IGE}}{\bar x}$
are all cones (remember Remark \ref{rem:Hucone}(i)),
it is possible to assume without any loss of generality that $|v|=1$.
According to the characterization of elements in the contingent cone
mentioned in Remark \ref{rem:charcontcone}, in order to prove that
$v\in\Tang{\Solv{\IGE}}{\bar x}$ it suffices to show that
\begin{equation}     \label{eq:Dlddist0}
  \liminf_{t\downarrow 0}{\dist{\bar x+tv}{\Solv{\IGE}}\over t}=
  \sup_{\tau>0}\inf_{t\in (0,\tau)}{\dist{\bar x+tv}{\Solv{\IGE}}\over t}
  =0.
\end{equation}
This means that for every $\tau>0$ and $\epsilon>0$, there must exist
$t\in (0,\tau)$ such that
\begin{equation}   \label{in:tangrapthesis}
  {\dist{\bar x+tv}{\Solv{\IGE}}\over t}\le\epsilon.
\end{equation}
So, fix positive $\tau$ and $\epsilon$. According to Theorem \ref{thm:erbometincr},
by virtue of hypotheses (i) and (ii), a local error bound for $(\IGE)$ is
in force, so corresponding to $\alpha\in (1,\incr{F}{S}{\bar x})$ there exists
$\delta_\alpha>0$ such that inequality $(\ref{in:locerboCincr})$ holds.

On the other hand, by virtue of hypothesis (iii), corresponding to $\epsilon$
there exists $\delta_\epsilon>0$ such that
\begin{equation}    \label{in:outprebarxtr}
  F(x)\subseteq F(\bar x)+H(x-\bar x)+\epsilon(\alpha-1)|x-\bar x|\Uball,
  \quad\forall x\in\ball{\bar x}{\delta_\epsilon}.
\end{equation}
Now, take $\delta_*$ in such a way that
$$
   0<\delta_*<\min\{\delta_\alpha,\, \delta_\epsilon,\, \tau\}.
$$
Since $v\in\WIang{S}{\bar x}$ there exists $t_*\in (0,\delta_*)$
with the property that $\bar x+t_*v\in S\cap\ball{\bar x}{\delta_*}$ .
As a consequence of inclusion $(\ref{in:outprebarxtr})$, taking
into account that $v\in H^\upinv(C)$, one finds
\begin{eqnarray*}
   F(\bar x+t_*v) &\subseteq& F(\bar x)+t_*H(v)+\epsilon(\alpha-1) t_*\Uball
   \subseteq C+t_*C+\epsilon(\alpha-1) t_*\Uball  \\
   &=& C+\epsilon(\alpha-1) t_*\Uball.
\end{eqnarray*}
From the last inclusion, on account of what recalled in Remark \ref{rem:excbehave}(iii),
it follows
$$
  \exc{F(\bar x+t_*v)}{C}\le\exc{C+\epsilon(\alpha-1) t_*\Uball}{C}=
  \exc{\epsilon(\alpha-1) t_*\Uball}{C}\le\epsilon(\alpha-1) t_*.
$$
Therefore, by exploiting the error bound inequality $(\ref{in:locerboCincr})$,
what is possible to do inasmuch as $\bar x+t_*v\in\ball{\bar x}{\delta_\alpha}
\cap S$, one obtains
$$
   {\dist{\bar x+t_*v}{\Solv{\IGE}}\over t_*}\le {\exc{F(\bar x+t_*v)}{C}
   \over (\alpha-1)t_*}\le\epsilon.
$$
As the last inequality shows that condition $(\ref{in:tangrapthesis})$ is
satisfied for $t=t_*\in (0,\tau)$, inclusion $(\ref{in:itanapproxsol})$
is proved.

In order to prove the second inclusion in the thesis, observe first that, since
the function $x\mapsto\dist{x}{\Solv{\IGE}}$ is Lipschitz, it holds
$$
  \liminf_{t\downarrow 0}{\dist{\bar x+tv}{\Solv{\IGE}}\over t}=
  \displaystyle\liminf_{w\to v\atop t\downarrow 0}
  {\dist{\bar x+tw}{\Solv{\IGE}}\over t}.
$$
By consequence, in order to show that if $v\in H^\upinv(C)\cap\Tang{S}{\bar x}$
then $v\in \Tang{\Solv{\IGE}}{\bar x}$ by means of the characterization in $(\ref{eq:Dlddist0})$,
it suffices to prove the existence of sequences $(v_n)_n$, with $v_n\to  v$,
and $(t_n)_n$, with $t_n\downarrow 0$, as $n\to\infty$, such that
\begin{equation}     \label{eq:limDlddist0}
  \lim_{n\to\infty}{\dist{\bar x+t_nv_n}{\Solv{\IGE}}\over t_n}=0.
\end{equation}
Again, one can assume that $|v|=1$ (the case $v=\nullv$ being trivial).
Since $v\in H^\upinv(C)\cap\Tang{S}{\bar x}$, there exist $(v_n)_n$,
with $v_n\to  v$, and $(t_n)_n$, with $t_n\downarrow 0$, such that
$\bar x+t_nv_n\in S$, for every $n\in\N$.
As a consequence of hypothesis (iv), one finds that for some $\kappa>0$
it must hold
$$
  H(v_n)\subseteq H(v)+\kappa|v_n-v|\Uball,\quad\forall n\in\N.
$$
Fix $\epsilon>0$.
Correspondingly, by the hypothesis (iii) there exists $\delta_\epsilon>0$
such that the following inclusion holds true
\begin{equation}    \label{in:outprederakeps}
  F(x)\subseteq F(\bar x)+H(x-\bar x)+|x-\bar x|
  \left({\alpha-1\over \kappa+2}\right)\epsilon\Uball,\quad
  \forall x\in\ball{\bar x}{\delta_\epsilon}.
\end{equation}
Take $\delta_*\in (0,\min\{\delta_\alpha,\, \delta_\epsilon\})$, where
$\delta_\alpha>0$ and $\alpha$ have the same meaning as in the first part of the proof
and do exist by hypotheses (i) and (ii) and by Theorem \ref{thm:erbometincr}.
Since $\bar x+t_nv_n\to\bar x$ as $n\to\infty$, there exists
$n_*\in\N$, such that
$$
  \bar x+t_nx_n\in \ball{\bar x}{\delta_*},
$$
and
$$
   |v_n-v|<{(\alpha-1)\epsilon\over\kappa+2},\qquad |v_n|<2,\quad
   \forall n\in\N, \ n\ge n_*.
$$
Thus, by recalling that $v\in H^\upinv(C)$, in the light of inclusion
$(\ref{in:outprederakeps})$, which can be used because $\bar x+t_nv_n\in
\ball{\bar x}{\delta_\epsilon}$ for every $n\ge n_*$, one obtains
\begin{eqnarray*}
  F(\bar x+t_nv_n) &\subseteq & F(\bar x)+t_n H(v_n)+t_n|v_n|\left({\alpha-1\over \kappa+2}\right)\epsilon\Uball \\
    &\subseteq & C+t_n[H(v)+\kappa|v_n-v|\Uball]+t_n|v_n|\left({\alpha-1\over \kappa+2}\right)\epsilon\Uball \\
    &\subseteq & C+t_nC+t_n{\kappa\over \kappa+2}(\alpha-1)\epsilon\Uball+t_n{2\over \kappa+2}(\alpha-1)\epsilon\Uball \\
    &=& C+t_n \left({\kappa\over \kappa+2}+{2\over \kappa+2}\right)(\alpha-1)\epsilon\Uball,\\
    &=& C+t_n(\alpha-1)\epsilon\Uball,\quad\forall  n\in\N,\ n\ge n_*.
\end{eqnarray*}
Now, by passing to the excess function, from the last inclusions one deduces
\begin{eqnarray*}
  \exc{F(\bar x+t_nv_n)}{C} &\le& \exc{C+t_n(\alpha-1)\epsilon\Uball}{C}
  \le\exc{t_n(\alpha-1)\epsilon\Uball}{C}  \\
  &\le & t_n(\alpha-1)\epsilon,\quad \forall n\in\N,\ n\ge n_*.
\end{eqnarray*}
Since $\bar x+t_nx_n\in \ball{\bar x}{\delta_*}\cap S$ for every $n\in\N$, with
$n\ge n_*$, by virtue of the error bound inequality valid in $\ball{\bar x}{\delta_\alpha}
\cap S$, it results in
$$
  {\dist{\bar x+t_nx_n}{\Solv{\IGE}}\over t_n}\le
  {\exc{F(\bar x+t_nv_n)}{C}\over (\alpha-1)t_n}\le
  \epsilon,\quad\forall  n\in\N,\ n\ge n_*.
$$
The last inequality, by arbitrariness of $\epsilon$, allows one
to conclude that equality $(\ref{eq:limDlddist0})$ holds true,
thereby completing the proof.
\end{proof}

Inclusions $(\ref{in:itanapproxsol})$ and $(\ref{in:tangapproxsol})$
provide a convenient description of (in the general case) some elements
in $\Tang{\Solv{\IGE}}{\bar x}$. Theorem \ref{thm:itanapproxsol}
ensures that, as far as working with solutions of the approximated
(actually, homogenized) set-inclusive generalized equation
$$
  \hbox{Find $x\in\WIang{S}{\bar x}$ such that } H(x)\subseteq C,
$$
one keeps within the conic (contingent) approximation of
$\Solv{\IGE}$ near $\bar x$. The reader should notice that, very
often, finding all solutions of problem $(\IGE)$ turns out to be
a hard problem. Consequently, the set $\Tang{\Solv{\IGE}}{\bar x}$ can not be
calculated explicitly. On the other hand, since $S$ and $C$ are problem
data, while the structure of $H$ is supposed to be simpler than the one
of $F$, cones $H^\upinv(C)$ and $\WIang{S}{\bar x}$, or $\Tang{S}{\bar x}$,
can be calculated more easily. This fact takes major evidence when
$H$ is a fan generated by linear mappings and, in particular, is finitely
generated (remember indeed Remark \ref{rem:Hucone}(ii)).
With such a reading, Theorem \ref{thm:itanapproxsol} can be
considered as a modern version of an implicit function theorem.

Since outer prederivatives are only one-side approximation tools, one
can not expect that any inclusion achieved through them, such as
$(\ref{in:itanapproxsol})$, could be reverted to get an equality.
A simple counterexample is discussed below.

\begin{example}[Strict inclusion may hold]
Let us consider the set-valued mapping $F:\R\rightrightarrows\R^2$
introduced in Example \ref{ex:glometCincrmap}. Take $S=\R$, $C=\R^2_+$
and $\bar x=0$. As $F$ is globally metrically $\R^2_+$-increasing, it
is metrically $\R^2_+$-increasing relative to $\R$ around $0$.
It is plain to check that $F$ is l.s.c. and Hausdorff $\R^2_+$-u.s.c.
on $\R$.  From Definition \ref{def:prederiv}(i) it follows that
the constant mapping $H:\R\rightrightarrows\R^2$ defined by $H(x)=\R^2$
for every $x\in\R$, is an outer prederivative of $F$ at $0$.
As one readily sees, it holds $\Solv{\IGE}=F^\upinv(\R^2_+)=[0,+\infty)$.
Thus, it results in $\Tang{\Solv{\IGE}}{0}=[0,+\infty)$. On the other
hand, it is clear that $H^\upinv(\R^2_+)=\varnothing$.
So, in the current case it happens
$$
  H^\upinv(\R^2_+)\cap \WIang{S}{0}=\varnothing\ne[0,+\infty)=
  \Tang{\Solv{\IGE}}{0}.
$$
Now, to work with a more reasonable approximation of $F$ at $0$,
one may consider the set-valued mapping $H:\R\rightrightarrows\R^2$,
defined by
$$
  H(x)=(x,x)+O,
$$
where $O=\{y=(y_1,y_2)\in\R^2:\ y_1y_2=0\}$. Clearly, $H$ is p.h.,
because, as $O$ is a cone, it holds
$$
  H(\lambda x)=(\lambda x,\lambda x)+O=\lambda(x,x)+\lambda O=
  \lambda[(x,x)+O]=\lambda H(x),\quad\forall \lambda>0,\
  \forall x\in\R.
$$
Moreover, since it is
$$
  F(x)\subseteq H(x),\quad\forall x\in\R,\quad\hbox{ and }\quad
  (0,0)\in F(0),
$$
one has for every $\epsilon>0$
$$
  F(x)\subseteq F(0)+H(x)\subseteq F(0)+H(x)+\epsilon|x|\Uball,
  \quad\forall x\in\R.
$$
Consequently, $H$ is an outer prederivative of $F$ at $0$, so all
the hypotheses of Theorem \ref{thm:itanapproxsol} are fulfilled.
Since
$$
  H(x)\not\subseteq\R^2_+,\quad\forall x\in\R,
$$
it happens that $H^\upinv(\R^2_+)=\varnothing$. So, again one has
$$
  H^\upinv(\R^2_+)\cap \WIang{S}{0}=\varnothing\ne[0,+\infty)=
  \Tang{\Solv{\IGE}}{0}.
$$
\end{example}

\begin{corollary}
With reference to problem $(\IGE)$, let $\bar x\in\Solv{\IGE}$. Suppose
that:

(i) $F$ is l.s.c. in a neighbourhood of $\bar x$ and Hausdorff $C$-u.s.c.
at $\bar x$;

(ii) $F$ admits a strict prederivative $H:\R^n\rightrightarrows\R^m$
at $\bar x$;

(iii) there exist $\eta>0$ and $\delta>0$ such that
\begin{equation*}
   \forall x\in\ball{\bar x}{\delta}\cap S\
   \exists u\in\Uball\cap\WIang{S}{x}:\ H(u)+\eta\Uball\subseteq C.
\end{equation*}
\noindent Then, inclusion $(\ref{in:itanapproxsol})$ holds true. Moreover,
if $H$ is Lipschitz, inclusion $(\ref{in:tangapproxsol})$ holds.
\end{corollary}

\begin{proof}
In the light of Proposition \ref{pro:metincrsopred}, by hypotheses
(ii) and (iii) $F$ turns out to be metrically $C$-increasing around
$\bar x$, relative to $S$. Then, it suffices to apply Theorem
\ref{thm:itanapproxsol}.
\end{proof}

Besides an inner tangential approximation  of the solution set
to $(\IGE)$, already useful in applications to optimization (see
Section \ref{Sect4}), it seems to be worthwhile to consider also an
outer tangential approximation of this set. In doing so, the following
remark is relevant.

\begin{remark}     \label{rem:FbarxiintC}
Under the assumption that $F$ is Hausdorff $C$-u.s.c. at $\bar x$,
which seems to be reasonable for the problem at the issue, if there exists
$\eta>0$ such that $F(\bar x)+\eta\Uball\subseteq C$ (strong satisfaction
of the set-inclusion), then one has $\bar x\in\inte F^\upinv(C)$. Indeed,
corresponding with $\eta$, there exists $\delta>0$ such that
$$
   F(x)\subseteq F(\bar x)+C+\eta\Uball\subseteq C+C=C,\quad\forall
   x\in\ball{\bar x}{\delta}.
$$
Therefore, whenever it happens that $\bar x\in\inte S$, one obtains
$$
  \bar x\in\inte F^\upinv(C)\cap\inte S=\inte\Solv{\IGE}.
$$
Consequently, it results in
$$
 \Tang{\Solv{\IGE}}{\bar x}=\R^n.
$$
In such a circumstance, an outer description of the contingent cone
loses interest.
\end{remark}

In the light of Remark \ref{rem:FbarxiintC}, the below analysis is focussed
on the case $F(\bar x)\cap\bd C\ne\varnothing$. Such a choice leaves out
the case $F(\bar x)\subseteq\inte C$, whenever $F(\bar x)$ is not a
compact set, which is a more general circumstance than that considered
in Remark \ref{rem:FbarxiintC} (strong satisfaction of the set-inclusion).

\begin{theorem}[Outer tangential approximation by fans]    \label{thm:otanapproxsol}
With reference to problem $(\IGE)$, let $\bar x\in\Solv{\IGE}$. Suppose
that:

(i) $F(\bar x)\cap\bd C\ne\varnothing$;

(ii) $F$ admits an inner prederivative $H:\R^n\rightrightarrows\R^m$ at $\bar x$;

(iii) $H$ is a fan generated by a bounded set ${\mathcal G}\subseteq\Lin(\R^n,\R^m)$.

\noindent Then, it holds
$$
  \Tang{\Solv{\IGE}}{\bar x}\subseteq\left[\bigcap_{y\in F(\bar x)\cap\bd C}
  H^\upinv(\Tang{C}{y})\right]\cap\Tang{S}{\bar x}.
$$
\end{theorem}

\begin{proof}
It is clear that
$$
  \nullv\in\left[\bigcap_{y\in F(\bar x)\cap\bd C}
  H^\upinv(\Tang{C}{y})\right]\cap\Tang{S}{\bar x}.
$$
Indeed, $\nullv\in\Tang{S}{\bar x}$ and, since $H$ is generated by linear
mappings, $H(\nullv)=\{\nullv\}$, with the consequence that $\nullv\in
H^\upinv(\Tang{C}{y})$, for every $y\in F(\bar x)\cap\bd C$.
Now, let $v\ne\nullv$ be an arbitrary element of $\Tang{\Solv{\IGE}}{\bar x}$.
As already done above, in consideration of the conical nature of all involved
sets, it is possible to assume that $|v|=1$. Then, there exist $(v_n)_n$, with
$v_n\to v$, and $(t_n)_n$, with $t_n\downarrow 0$, as $n\to\infty$, such that
$\bar x+t_nv_n\in\Solv{\IGE}=F^\upinv(C)\cap S$, for every $n\in\N$.
This fact immediately implies that $v\in\Tang{S}{\bar x}$.
By virtue of hypothesis (ii), for every $\epsilon>0$ there exists $\delta_\epsilon>0$
such that
\begin{equation}     \label{in:useinprederx}
  F(\bar x)+H(x-\bar x)\subseteq F(x)+\epsilon|x-\bar x|\Uball,\quad
  \forall x\in\ball{\bar x}{\delta_\epsilon}.
\end{equation}
According to what was noted in Remark \ref{rem:Hucone}(iii), since
$H$ is generated by a bounded set it is Lipschitz, so there exists
$\kappa>0$ such that
\begin{equation}     \label{in:Hlipkappa}
   H(x_1)\subseteq H(x_2)+\kappa|x_1-x_2|\Uball,\quad\forall
   x_1,\, x_2\in\R^n.
\end{equation}
Fix an element $\bar y\in F(\bar x)\cap\bd C$ and $\epsilon>0$. Without any
loss of generality, it is possible to assume that
$$
  \delta_\epsilon<{\epsilon\over\kappa}.
$$
Since the sequence
$(\bar x+t_nv_n)_n$ converges to $\bar x$, starting with a proper $n_*\in\N$
it must be $\bar x+t_nv_n\in\ball{\bar x}{\delta_\epsilon}$ for every
$n\ge n_*$. Thus, from inclusion $(\ref{in:useinprederx})$ it follows
$$
  \bar y+t_nH(v_n)\subseteq F(\bar x+t_nv_n)+\epsilon t_n|v_n|\Uball,
  \quad\forall n\ge n_*,
$$
whence, by recalling that $\bar x+t_nv_n\in\Solv{\IGE}$, one gets
$$
  t_nH(v_n)\subseteq F(\bar x+t_nv_n)-\bar y+\epsilon t_n|v_n|\Uball
  \subseteq C-\bar y+\epsilon t_n|v_n|\Uball,\quad\forall n\ge n_*.
$$
Since it is $v_n\to v$, by increasing, if needed, the value of $n_*$,
it is
$$
   |v_n|<2 \qquad\hbox{ and }\qquad |v-v_n|<{\epsilon\over\kappa},
   \quad\forall n\ge n_*.
$$
From the last inclusion, one obtains
\begin{equation}    \label{in:HvnCepsB}
   H(v_n)\subseteq {C-\bar y\over t_n}+2\epsilon\Uball,\quad
   \forall n\ge n_*.
\end{equation}
By recalling inclusion $(\ref{in:Hlipkappa})$, one deduces
$$
  H(v)\subseteq H(v_n)+\kappa|v-v_n|\Uball\subseteq H(v_n)
  +\epsilon\Uball,\quad\forall n\ge n_*.
$$
On account of $(\ref{in:HvnCepsB})$, the last obtained inclusion gives
$$
  H(v)\subseteq {C-\bar y\over t_n}+3\epsilon\Uball,
  \quad\forall n\ge n_*.
$$
According to this, for each $w\in H(v)$ there exist sequences
$(c_n)_n$, with $c_n\in C$, and $(b_n)_n$, with $b_n\in\Uball$,
such that
$$
   w={c_n-\bar y\over t_n}+3\epsilon b_n, \quad\forall n\ge n_*.
$$
Since, up to a sequence relabeling, it is $b_n\longrightarrow b
\in\Uball$, for some $b\in\Uball$ as $n\to\infty$, $\Uball$ being
compact, it must result in
$$
  z_n={c_n-\bar y\over t_n}\longrightarrow z\in\Tang{C}{\bar y}
  \hbox{ as }  n\to\infty.
$$
As it is $\bar y+t_nz_n\in C$ for every $n\ge n_*$, this
means that $w\in\Tang{C}{\bar y}+3\epsilon\Uball$. By arbitrariness
of $w\in H(v)$, the above argument shows that
\begin{equation}      \label{in:Htangeps}
   H(v)\subseteq \Tang{C}{\bar y}+3\epsilon\Uball.
\end{equation}
Since $H(v)$ is a closed set, $\Tang{C}{\bar y}$ is a closed cone and
inclusion $(\ref{in:Htangeps})$ remains true for every $\epsilon>0$
(notice, indeed, that $v$ has been fixed before fixing $\epsilon$), then according
to what noted in Remark \ref{rem:vectprops}(iv) it is possible to
assert that $H(v)\subseteq \Tang{C}{\bar y}$, or, equivalently,
$v\in H^\upinv(\Tang{C}{\bar y})$. By arbitrariness
of $\bar y\in F(\bar x)\cap\bd C$, the above argument shows that
$$
  v\in\bigcap_{y\in F(\bar x)\cap\bd C}
  H^\upinv(\Tang{C}{y}),
$$
and thereby allows one to conclude that $\Tang{\Solv{\IGE}}{\bar x}\subseteq
\bigcap_{y\in F(\bar x)\cap\bd C}H^\upinv(\Tang{C}{y})$.
The proof is complete.
\end{proof}

In the special case in which $\nullv\in F(\bar x)$, by exploiting the bilateral
approximation of a set-valued mapping provided by prederivatives, one can
achieve the following characterization on the contingent cone to the
solution set of a $(\IGE)$.

\begin{theorem}[Tangential approximation of $\Solv{\IGE}$]     \label{thm:tanapproxsol}
With reference to problem $(\IGE)$, let $\bar x\in\Solv{\IGE}$. Suppose
that:

(i) $F$ is l.s.c. in a neighbourhood of $\bar x$ and Hausdorff $C$-u.s.c.
at $\bar x$;

(ii) $F$ is metrically $C$-increasing around $\bar x$ relative to $S$;

(iii) $F$ admits a prederivative $H:\R^n\rightrightarrows\R^m$ at $\bar x$;

(iv) $H$ is a fan  generated by a bounded set;

(v) $\nullv\in F(\bar x)$.

\noindent Then, the following equality holds
\begin{equation}    \label{eq:tanapproxsol}
  \Tang{\Solv{\IGE}}{\bar x}=H^\upinv(C)\cap\Tang{S}{\bar x}.
\end{equation}
\end{theorem}

\begin{proof}
Under the above hypotheses one can invoke Theorem \ref{thm:itanapproxsol}.
In doing so, as $H$ is generated by a bounded set of linear mappings,
it is a Lipschitz  outer prederivative of $F$ at $\bar x$. Consequently,
inclusion $(\ref{in:tangapproxsol})$ must hold true.

On the other hand, the hypotheses in force allows one to apply Theorem
\ref{thm:otanapproxsol}. Thus, since $\nullv\in F(\bar x)\cap\bd C$
and $\Tang{C}{\nullv}=C$, one finds
$$
  \Tang{\Solv{\IGE}}{\bar x}\subseteq H^\upinv(\Tang{C}{\nullv})
  \cap\Tang{S}{\bar x}=H^\upinv(C)\cap\Tang{S}{\bar x}.
$$
The last inclusion, along with $(\ref{in:tangapproxsol})$, certifies that
the equality in the assertion is true.
\end{proof}

It is reasonable to expect that, owing to the local nature of the
contingent tangential approximation, in the case $\bar x\in\inte S$
the presence of $S$ does not affect the representation of
$\Tang{\Solv{\IGE}}{\bar x}$. This is established below.

\begin{corollary}
Under the hypotheses of Theorem \ref{thm:tanapproxsol}, suppose that
$\bar x\in\inte S$. Then, it holds
\begin{equation}    \label{eq:tangapproxsolinte}
  \Tang{F^\upinv(C)}{\bar x}=\Tang{\Solv{\IGE}}{\bar x}=H^\upinv(C).
\end{equation}
\end{corollary}

\begin{proof}
Since $\bar x$ is an interior point of $S$, there exists $\delta_0>0$
such that $\ball{\bar x}{\delta_0}\subseteq S$. By taking into
account equality $(\ref{eq:locbehaTang})$, one obtains
\begin{eqnarray*}
  \Tang{\Solv{\IGE}}{\bar x} &=& \Tang{S\cap F^\upinv(C)}{\bar x}=
  \Tang{(S\cap F^\upinv(C))\cap\ball{\bar x}{\delta_0}}{\bar x}  \\
  &=&\Tang{F^\upinv(C)\cap\ball{\bar x}{\delta_0}}{\bar x}=
  \Tang{F^\upinv(C)}{\bar x}.
\end{eqnarray*}
On the other hand, again by the fact that $\bar x\in\inte S$, it is
$\Tang{S}{\bar x}=\R^n$. Thus, in the current case $(\ref{eq:tanapproxsol})$
becomes $(\ref{eq:tangapproxsolinte})$.
\end{proof}

\vskip1cm


\section{Applications to constrained optimization} \label{Sect4}

In the present section, the tangential analysis of the solution set
to set-inclusive generalized equation is exploited for deriving
necessary optimality conditions. Let us focus on constrained scalar
optimization problems that can be formalized as
$$
 \min_{x\in S}\varphi(x)\quad \hbox{ subject to }\quad F(x)\subseteq C,
 \leqno (\OP)
$$
where $\varphi:\R^n\longrightarrow\R$ denotes the objective (or cost)
function, while the sets $S\subseteq\R^n$ and $C\subseteq\R^m$, and the
set-valued mapping $F:\R^n\rightrightarrows\R^m$, define a constraint
system leading to a set-inclusive generalized equation. With this format,
the feasible region of the problem is therefore ${\mathcal R}=\Solv{\IGE}=
F^\upinv(C)\cap S$. As in the previous sections, $S$ is assumed to be a
nonempty closed set whereas $C$ a closed, convex and pointed cone.

\begin{proposition}[Necessary optimality condition]    \label{pro:gennoc}
Let $\bar x\in{\mathcal R}$ be a local solution to problem $(\OP)$.
Suppose that:

(i) $F$ is l.s.c. in a neighbourhood of $\bar x$ and Hausdorff $C$-u.s.c.
at $\bar x$;

(ii) $F$ is metrically $C$-increasing around $\bar x$ relative to $S$;

(iii) $F$ admits $H:\R^n\rightrightarrows\R^m$ as an outer prederivative
at $\bar x$.

\noindent Then, the following inclusion holds
\begin{equation}   \label{in:locoptnc}
    -\FuD\varphi(\bar x)\subseteq\ndc{\left[H^\upinv(C)\cap\WIang{S}{\bar x}
    \right]}.
\end{equation}
If, in particular, $\varphi$ is differentiable at $\bar x$, condition
$(\ref{in:locoptnc})$ becomes
$$
  \nullv\in\nabla\varphi(\bar x)+\ndc{\left[H^\upinv(C)\cap\WIang{S}{\bar x}
    \right]}.
$$
\end{proposition}

\begin{proof}
By the local optimality of $\bar x$, there exists $\delta>0$ such that
$$
   \varphi(\bar x)\le\varphi(x),\quad\forall x\in{\mathcal R}
   \cap\ball{\bar x}{\delta}.
$$
Take an arbitrary $w\in\FuD\varphi(\bar x)$.
According to what recalled in Remark \ref{rem:varformFuD},
there exists $\sigma:\R^n\longrightarrow\R$, such that
\begin{equation}     \label{in:locoptsigma}
   \sigma(\bar x)=\varphi(\bar x)\le\varphi(x)\le\sigma(x),
   \quad\forall x\in{\mathcal R}\cap\ball{\bar x}{\delta},
\end{equation}
with $\nabla\sigma(\bar x)=w$. Take $v\in (H^\upinv(C)\cap\WIang{S}{\bar x})
\backslash\{\nullv\}$.
By virtue of inclusion $(\ref{in:itanapproxsol})$, that holds true because
all hypotheses of Theorem \ref{thm:itanapproxsol} are in force, there
must exist sequences $(v_n)_n$, with $v_n\to v$ and $(t_n)_n$, with
$t_n\downarrow 0$, such that $\bar x+t_nv_n\in {\mathcal R}$ for every
$n\in\N$. Since $\bar x+t_nv_n\to\bar x$ as $n\to\infty$, there exists
a proper $n_*\in\N$ such that
$$
   \bar x+t_nv_n\in {\mathcal R}\cap\ball{\bar x}{\delta},\quad
   \forall n\ge n_*,\ n\in\N.
$$
Thus, from inequality $(\ref{in:locoptsigma})$, by using the
differentiability of $\sigma$ at $\bar x$, one obtains
$$
  0\le {\sigma(\bar x+t_nv_n)-\sigma(\bar x)\over t_n}=
  \langle w,v_n\rangle+{o(|t_nv_n|)\over t_n|v_n|}\cdot|v_n|,
  \quad\forall  n\ge n_*,\ n\in\N.
$$
Take into account that, as a converging sequence, $(v_n)_n$ is
bounded. So, passing to the limit as $n\to\infty$ in the last inequality,
one finds
$$
  \langle w,v\rangle\ge 0.
$$
As this is true for every $v\in H^\upinv(C)\cap\WIang{S}{\bar x}$
(the case $v=\nullv$ being trivial), one can deduce that
$$
  -w\in\ndc{\left[H^\upinv(C)\cap\WIang{S}{\bar x}
    \right]}.
$$
By arbitrariness of $w\in\FuD\varphi(\bar x)$, the last inclusion
gives $(\ref{in:locoptnc})$. The second assertion in the thesis
trivially follows.
\end{proof}

\begin{remark}      \label{rem:constsysinte}
It is to be noted that, whenever $\bar x\in\inte S$ satisfies the
constraint system in a ``strict" way, i.e. the set-inclusive generalized
equation is strongly satisfied at $\bar x$ in the sense of Remark \ref{rem:FbarxiintC},
then under Hausdorff the upper semicontinuity assumption on $F$ one has $\bar x
\in\inte{\mathcal R}$. In such a circumstance, the local optimality of $\bar x$
clearly implies $\nullv\in\FlD\varphi(\bar x)$.
\end{remark}

The optimality condition formulated in Proposition \ref{pro:gennoc} requests
$F$ to admit an outer prederivative $H$, but does not impose specific
requirements on $H$ (all hypotheses refer indeed to $F$).
As one expects, by adding proper assumptions on the geometric structure
of $H$, along with adequate qualification conditions, it is possible
to achieve finer optimality conditions, having a stronger computational
impact. This is done in the next result.

\begin{theorem}     \label{thm:noptcond}
Let $\bar x\in{\mathcal R}$ be a local solution to problem $(\OP)$.
Suppose that:

(i) $F$ is l.s.c. in a neighbourhood of $\bar x$ and Hausdorff $C$-u.s.c.
at $\bar x$;

(ii) $F$ admits a strict prederivative $H:\R^n\rightrightarrows\R^m$
at $\bar x$;

(iii) $H$ is a fan generated by a bounded set ${\mathcal G}\subseteq
\Lin(\R^n;\R^m)$;

(iv) there exist $\eta,\, \delta>0$ such that for every $x\in\ball{\bar x}{\delta}
\cap S$ exists $u\in\Uball\cap\WIang{S}{x}:\ H(u)+\eta\Uball\subseteq C$;

(v) $S$ is locally convex around $\bar x$ and it holds $\inte\Tang{S}{\bar x}
\cap\inte\left(\displaystyle\bigcap_{\Lambda\in{\mathcal G}}\Lambda^{-1}(C)\right)
\ne\varnothing$.

\noindent Then, it holds
\begin{equation}     \label{in:fannoc}
  -\FuD\varphi(\bar x)\subseteq \ndc{\left(\bigcap_{\Lambda\in{\mathcal G}}
  \Lambda^{-1}(C)\right)}+\Ncone{S}{\bar x}.
\end{equation}
\noindent If, in particular, it is

(vi) ${\mathcal G}=\conv\{\Lambda_1,\dots,\Lambda_p\}$ and $\cap_{i=1}^p
\inte\Lambda_i^{-1}(C)\ne\varnothing$,

\noindent  then inclusion $(\ref{in:fannoc})$ becomes
\begin{equation}    \label{in:multrulenS}
   -\FuD\varphi(\bar x)\subseteq \sum_{i=1}^p\Lambda_i^\top(\ndc{C})+
   \Ncone{S}{\bar x}.
\end{equation}
\end{theorem}

\begin{proof}
Observe first that, by virtue of hypotheses (i), (ii) and (iv), according to Proposition
\ref{pro:metincrsopred} $F$ turns out to be metrically $C$-increasing around
$\bar x$, relative at $S$. Besides, by virtue of hypothesis (iii), the fan $H$ is
Lipschitz (remember Remark \ref{rem:Hucone}(iii)). Thus,  Theorem \ref{thm:itanapproxsol}
ensures that the inclusion
$$
   H^\upinv(C)\cap\Tang{S}{\bar x}\subseteq\Tang{\mathcal R}{\bar x}
$$
holds true. By reasoning exactly as in the proof of Proposition \ref{pro:gennoc},
one finds
$$
  -\FuD\varphi(\bar x)\subseteq\ndc{\left[H^\upinv(C)\cap\Tang{S}{\bar x}
    \right]}.
$$
Now, since $H^\upinv(C)=\bigcap_{\Lambda\in{\mathcal G}}\Lambda^{-1}(C)$ and
$\Tang{S}{\bar x}$ are closed cones satisfying the qualification condition
in hypothesis (v), by what recalled in Remark \ref{rem:ndccalcul} (see, in particular,
formula $(\ref{eq:ndcintesecteq})$), one obtains
$$
    -\FuD\varphi(\bar x)\subseteq\ndc{\left(\bigcap_{\Lambda\in{\mathcal G}}
  \Lambda^{-1}(C)\right)}+\ndc{\left(\Tang{S}{\bar x}\right)},
$$
which gives inclusion $(\ref{in:fannoc})$, on account of the relation between
the (negative) dual cone of the contingent cone and the normal cone.

Under the additional hypothesis (vi), the further qualification condition allows
one to exploit formula $(\ref{eq:invadj})$, in such a way to obtain
$$
   \ndc{\left(\bigcap_{\Lambda\in{\mathcal G}}\Lambda^{-1}(C)\right)}=
   \ndc{\left(\bigcap_{i=1}^p\Lambda_i^{-1}(C)\right)}=
   \sum_{i=1}^{p}\ndc{\left(\Lambda_i^{-1}(C)\right)}=
   \sum_{i=1}^{p}\Lambda_i^\top(\ndc{C}).
$$
This completes the proof.
\end{proof}

\begin{corollary}    \label{cor:mulrul}
Let $\bar x\in{\mathcal R}$ be a local solution to problem $(\OP)$.
Suppose that:

(i) $F$ is l.s.c. in a neighbourhood of $\bar x$ and Hausdorff $C$-u.s.c.
at $\bar x$;

(ii) $F$ admits a strict prederivative $H:\R^n\rightrightarrows\R^m$
at $\bar x$;

(iii) $H$ is a fan finitely-generated by a set ${\mathcal G}=\conv
\{\Lambda_1,\dots,\Lambda_p\}$, with the property $\bigcap_{i=1}^p\inte\Lambda_i^{-1}(C)
\ne\varnothing$;

(iv) there exist $\eta>0$ and $u\in\Uball$ such that $H(u)+\eta\Uball\subseteq C$;

(v) $\bar x\in\inte S$;

(vi) $\varphi$ is differentiable at $\bar x$.

\noindent Then, there exist $y_1,\dots,y_p\in\R^m$ such that
$$
   y_i\in\ndc{C},\quad\forall i=1,\dots, p
$$
and
\begin{equation}     \label{eq:multrul}
  \nabla\varphi(\bar x)+\sum_{i=1}^{p}\Lambda_i^\top(y_i)=
  \nullv.
\end{equation}
\end{corollary}

\begin{proof}
Under hypothesis (v), there exists $\delta>0$ such that
$\ball{\bar x}{\delta}\subseteq S$ and for every $x\in\ball{\bar x}{\delta}$
it is $\ball{x}{\delta}\subseteq S$ too. As a consequence, one has that
$\WIang{S}{\bar x}=\Tang{S}{\bar x}=\R^n$ and the current hypothesis (iv) ensures
what is requested in hypothesis (iv) in Theorem \ref{thm:noptcond}. Then,
the thesis follows immediately from Theorem \ref{thm:noptcond}, if taking
into account that, in the present case, it is $\Ncone{S}{\bar x}=\{\nullv\}$
and $\FuD\varphi(\bar x)=\{\nabla\varphi(\bar x)\}$.
\end{proof}

The necessary optimality condition formulated in Corollary \ref{cor:mulrul}
might remind a multiplier rule, with elements $y_i$, $i=1,\dots,p$, playing
the role of multipliers. Nevertheless, in comparison with classical
Lagrangian type optimality conditions, some substantial differences
evidently  emerge.
Notice indeed that each $y_i$ is a vector of $\R^m$, not a scalar. Besides,
all terms $y_i$ refer to the same constraint $F(x)\subseteq C$. Their number
is given by the number of linear mappings needed to represent the strict outer
prederivative of $F$ at $\bar x$. So, it depends on the tool utilized for
approximating $F$ near $\bar x$, it is not an intrinsic constant of the
constraint system (and hence of the problem). On the other hand, the conditions
$y_i\in\ndc{C}$, $i=1,\dots,p$, can be regarded as a vector counterpart
of a sign condition, which is typical of optimality conditions for problems
with side-constraints (inequality systems and their generalizations).

As a further comment referring both, conditions $(\ref{in:multrulenS})$ and
$(\ref{eq:multrul})$, let us point out the computational appeal that
these conditions display: such a nontrivial constraint system as $(\IGE)$
turns out to be treated, under proper assumptions, by means of linear
algebra tools. This holds a fortiori whenever $C$ is polyhedral.

\vskip1cm


\end{document}